\documentclass{amsart}[10pt]
\usepackage{amsmath, bm, amsrefs, accents, amsfonts, amscd, amssymb, amsthm,
  verbatim, psfrag, dsfont,enumerate,url}
\usepackage{color, pdfsync}
\newtheorem{lemma}{Lemma}[section]

\newtheorem{theorem}{Theorem}
\newtheorem{corollary}[lemma]{Corollary}
\newtheorem{definition}[lemma]{Definition}
\newtheorem{remark}[lemma]{Remark}

\let\oldmarginpar\marginpar
\renewcommand\marginpar[1]{\-\oldmarginpar[\raggedright\footnotesize #1]%
{\raggedright\footnotesize #1}}

\begin{document}

\title{Exhaustive Gromov compactness for pseudoholomorphic curves}
\dedicatory{Dedicated to the Memory of Jean-Cristophe Yoccoz}

\author[J.W.~Fish]{Joel W.~Fish}
\thanks{This research was supported in part by NSF-DMS Standard Research Grant, Award 1610452 }

\author[H.~Hofer]{Helmut Hofer}

\address{
	Joel W.~Fish\\
	Mathematics Department\\
        University of Massachusetts Boston
	}

\email{joel.fish@umb.edu}

\address{
	Helmut Hofer\\
        School of Mathematics
	Institute for Advanced Study
	}

\email{hofer@math.ias.edu}

\keywords{pseudoholomorphic, compactness} 
\date{\today}
\maketitle

\begin{abstract}
Here we extend the notion of target-local Gromov convergence of
  pseudoholomorphic curves to the case in which the target manifold is not
  compact, but rather is exhausted by compact neighborhoods.
Under the assumption that the curves in question have uniformly bounded
  area and genus on each of the compact regions (but not necessarily
  global bounds), we prove a subsequence converges in an exhaustive
  Gromov sense.
\end{abstract}

\tableofcontents

\allowdisplaybreaks

\newcounter{CurrentSection}
\newcounter{CurrentTheorem}
\newcounter{CounterSectionEGC}
\newcounter{CounterTheoremEGC}

%%%%%%%%%%%%%%%%%%%%%%%%%%%%%%%%%%%%%%%%%%%%%%%%%%%%%%%%%%%%%%%%%%%%%%%%%%%%%%%%
%%%%%%%%%%                           SECTION                           %%%%%%%%%
%%%%%%%%%%%%%%%%%%%%%%%%%%%%%%%%%%%%%%%%%%%%%%%%%%%%%%%%%%%%%%%%%%%%%%%%%%%%%%%%
%

\section{Introduction}\label{SEC_introduction}
In his celebrated 1985 paper, \cite{Gr}, Gromov introduced the notion of a
  pseudoholomorphic curve, and provided an accompanying compactness theorem.
His idea was to generalize the notion of an algebraic curve in, say, a
  complex projective variety to that of a pseudoholomorphic curve in
  a symplectically tamed almost complex manifold, and he showed that families
  of such curves are analogously compact.
In the decades since, pseudoholomorphic curves have played a fundamental
  role in the development of symplectic geometry and topology as well as
  Hamiltonian dynamics, and a variety generalizing compactness theorems have
  been established.
These tend to proceed along two general paths.  

The first approach is exemplified by Rugang Ye \cite{Ye}, Floer
  \cite{Floer}, Hofer \cite{H93}, and the SFT compactness paper
  \cite{BEHWZ}, in that each of these treat closed or punctured curves
  from a global perspective.
Besides additional ingredients dealing with the analysis near punctures
  and the necks, see for example \cite{HWZ02},  the analysis proceeds rather
  analogously to that for families of harmonic maps, which we outline as
  follows.
  \begin{enumerate}                                                       %% NUM
    \item Obtain convergence of underlying Riemann surfaces.
    \item With respect to a constant curvature metric guaranteed by the
      Uniformization Theorem, show that gradient bounds imply
      \(\mathcal{C}^\infty\) bounds.
    \item Employ bubbling analysis at points of gradient blow-up, and show
      that only finitely many bubbles appear due to energy bounds and an
      energy threshold.
    \item Use \(\mathcal{C}^\infty\) bounds and Arzel\`{a}-Ascoli to pass
      to a further subsequence which converges in \(\mathcal{C}^\infty\).
    \item Verify that bubbles connect via, say, a monotonicity lemma.
    \end{enumerate}
This approach is most applicable when one has genus bounds, energy bounds,
  some global control over the entirety of each curve in the family, and
  when one already has a good idea of what types of curves should arise in
  the limits of such families.
For completeness, we also mention Hummel \cite{Hummel}, and for a more
  classical viewpoint, \cite{Audin}.

The second approach, typified by Taubes via Proposition 3.3 in
  \cite{Taubes}, treats curves as sets and integral currents, and proves
  compactness from a more measure theoretic perspective.
Roughly speaking, area bounds, a monotonicity lemma, and some measure
  theory yield a compactness theorem, however some additional work is
  necessary to show the limit is rectifiable, or rather that the measure
  theoretic limit has the structure of a weighted union of images of
  pseudoholomorphic curves.
This approach is quite natural from the perspective of Seiberg-Witten
  theory, particularly when employing a Taubes-like degeneration to obtain
  pseudoholomorphic curves.
The result has also been used extensively in Embedded Contact Homology,
  introduced by Hutchings in \cite{Hut1}; see also \cite{HS} and
  \cite{Hut2}.
More generally, the technique is applicable when one has little more than
  area bounds -- indeed, one does not need genus bounds on the sequence of
  curves.
However, this can also be a weakness, in that genus cannot be detected a
  priori by these techniques.  For example, one can construct degree-two
holomorphic branched coverings
  of the unit disk with arbitrarily large genus, but from the integral
  current perspective, all such objects are indistinguishable.

The purpose of this manuscript is to further develop a less used third
  approach, introduced by the first author in \cite{F0} and streamlined in
  \cite{F1} and \cite{F2}.
This is the so-called target-local Gromov compactness result; for a
  restatement, see Theorem \ref{THM_target_local} below.
The basic idea was to follow the Taubes approach to studying curves
  locally in the target and allowing a free boundary (including arbitrarily
  many boundary components), but also demanding a genus bound, and then
  extracting a subsequence which converges in the Gromov-topology, rather
  than the substantially weaker topology of integral-currents.
In some sense, the target-local compactness theorem says that if \(W\) is
  a smooth compact manifold with boundary, and \(u_k:(S_k,\partial S_k)\to
  (W, \partial W)\)  is a sequence of pseudoholomorphic maps with genus
  bounds and area bounds, then after trimming the curves near \(\partial W\)
  a subsequence converges in a Gromov sense.
Our main result, stated below as Theorem
  \ref{THM_exhaustive_gromov_compactness}, extends this to the case that
  \(W\) is no longer compact, but instead is exhausted by compact manifolds
  with smooth boundary:
  \begin{align*}                                                          %% EQN
    W_1 \subset W_2 \subset W_3 \subset \cdots \bigcup_{\ell\in\mathbb{N}}
    W_\ell = W.
    \end{align*}
Here the key assumptions on the curves are that 
  \begin{align*}                                                          %% EQN
    {\rm Area}(u_k^{-1}(W_\ell)) \leq C_\ell \qquad\text{and}\qquad{\rm
    Genus}(u_k^{-1}(W_\ell)) \leq C_\ell .
    \end{align*}
In other words, this means that the curves in question may have infinite
  area and genus, however on each compact \(W_\ell\) (the
  union of which exhaust \(W\)) one has area and genus bounds for
  the portions of curves in that region.

It is important to mention that our main result here, stated as Theorem
  \ref{THM_exhaustive_gromov_compactness} below, is not a needless
  extension of Theorem \ref{THM_target_local}, proved in \cite{F2}, but
  rather it plays a foundational role in two forthcoming papers.
The first was announced in \cite{FH0} and will appear in \cite{FH} in
  which we prove that no regular energy level of a
  proper Hamiltonian function on \((\mathbb{R}^4, \omega_{\rm std})\) has
  a minimal Hamiltonian flow, which answers a question for the case \(n=2\)
  raised by Herman in his 1998 ICM address; see \cite{Herman}.
The idea is to use neck-stretching techniques to study pseudoholomorphic
  curves in the symplectization of framed Hamiltonian manifolds.
The tremendous difficulty is that such curves will lack a priori energy
  bounds like those that appear in Symplectic Field Theory, and thus the
  global techniques employed in \cite{BEHWZ} fail quickly and completely.
Moreover, the Taubes approach of \cite{Taubes} also fails, precisely
  because the topology used to obtain compactness is simply too coarse.
Indeed, the genus bounds and curvature properties that follow from the
  Gromov topology (but not the integral current topology) which are
  guaranteed by Theorem \ref{THM_exhaustive_gromov_compactness} play
  crucial roles in the proofs of the main results in \cite{FH}.
In essence, our main result here strikes
  the perfect balance between the flexibility of the integral-current
  approach with the strength of the Gromov topology, and this balance is
  then heavily exploited in \cite{FH} to first find a limit curve (which
  might be wildly complicated), and then to use the Gromov topology and a
  posteriori analysis on the limit curve to show that it has a surprising
  number of unexpected properties, which are necessary to establish the
  non-minimality of the hypersurfaces.

The second result relying on Theorem
  \ref{THM_exhaustive_gromov_compactness} is the so called sideways
  stretching compactness results developed by the first author; see
  \cite{F3}.
Here the idea is that Symplectic Field Theory is something akin to a TQFT
  for symplectic manifolds, and an extended TQFT would be akin to an
  extended Symplectic Field Theory in which one could independently stretch
  the neck along two transverse contact hypersurfaces.
This has been carried out by the first author in certain sub-critical
  cases, and will appear in a forthcoming paper.
Again though, the idea is similar: use a sequence of expanding domains in
  the target manifold, on which one has successively increasing area bounds
  to obtain a preliminary compactness result from Theorem
  \ref{THM_exhaustive_gromov_compactness}; then use a posteriori analysis
  and the Gromov topology to improve properties of both the limit and
  precision of the convergence; then iterate this procedure to develop a
  full extended Symplectic Field Theory style compactness theorem.

More generally still, it is not difficult to imagine a wide range of
  applications of Theorem \ref{THM_exhaustive_gromov_compactness}.
Indeed, consider any symplectic manifold \(W\), and any compact set
  \(K\subset W\) which has empty interior.
Then consider any sequence of almost complex structures which are tame,
  but degenerate along \(K\).
That is, the \(J_k\) converge in \(\mathcal{C}_{loc}^\infty(W\setminus
  K)\) to an almost complex structure on \(W\setminus K\), which is
  uniformly tame on each \(W\setminus \mathcal{O}(K)\), but not uniformly
  tame on \(W\setminus K\).
Then consider a sequence of closed pseudoholomorphic curves in a fixed
  homology class in \(W\) which have bounded genus (e.g. only spheres).
Theorem \ref{THM_exhaustive_gromov_compactness} immediately guarantees
  that a subsequence converges, in an exhaustive Gromov sense (see
  Definition \ref{DEF_exhaustive_gromov_convergence}), to a
  pseudoholomorphic curve in \(W\setminus K\).
Such a curve may have wildly complicated behavior -- and yet a posteriori
  analysis can be employed which exploits the particular features of the
  \(K\) and \(J_k\) in question.
Considering the ubiquitous use of pseudoholomorphic curves in symplectic
  geometry, topology, and Hamiltonian dynamics, such a result would seem
  potentially quite useful.

\subsection{Acknowledgements}
The authors would like to thank the referee for helpful comments and
  suggestions, particularly in regards to clarifying the exposition and
  helping to remove certain unnecessary assumptions.

%%%%%%%%%%%%%%%%%%%%%%%%%%%%%%%%%%%%%%%%%%%%%%%%%%%%%%%%%%%%%%%%%%%%%%%%%%%%%%%%
%%%%%%%%%%                           SECTION                           %%%%%%%%%
%%%%%%%%%%%%%%%%%%%%%%%%%%%%%%%%%%%%%%%%%%%%%%%%%%%%%%%%%%%%%%%%%%%%%%%%%%%%%%%%
%
\section{Preliminaries}
  \label{SEC_preliminaries}

This section is devoted to presenting some preliminary concepts and small
  supporting results.
In Section \ref{SEC_direct_limit_manifolds} we introduce the notion of an
  embedding diagram and direct limit manifolds.
This is meant to generalize the notion of an exhausting sequence of
  regions:
  \begin{align*}                                                          %% EQN
    W_1 \subset W_2 \subset W_3 \subset \cdots \subset  \bigcup_{k\in
    \mathbb{N}} W_k =: \overline{W},
    \end{align*}
  and is necessary for constructing the domain of the limit curve we must
  later produce.
In Section \ref{SEC_riemann_surfaces} we review the basic definitions of
  Riemann surfaces, as well as additional structures like marked points,
  nodal points, decorations, arithmetic genus, etc.
In Section \ref{SEC_pseudoholomorphic}, we recall the definition of
  pseudoholomorphic curves and some related concepts, like stability,
  boundary-immersed maps, generally immersed maps, and area.
Finally, in Section \ref{SEC_convergence_of_pseudoholomorphic_curves} we
  provide a number of definitions of convergence for pseudoholomorphic
  curves.
We note that the key notion of Section \ref{SEC_preliminaries} is given in
  Definition \ref{DEF_exhaustive_gromov_convergence}, which is the novel
  definition of convergence of pseudoholomorphic curves in an exhaustive
  Gromov sense.
Finally we note that here and throughout, in the case that a domain is
  non-compact, \(\mathcal{C}^\infty\) convergence will mean
  \(\mathcal{C}_{loc}^\infty\) convergence.

%%%%%%%%%%%%%%%%%%%%%%%%%%%%%%%%%%%%%%%%%%%%%%%%%%%%%%%%%%%%%%%%%%%%%%%%%%%%%%%%
%%%%%%%%%%                         SUB-SECTION                         %%%%%%%%%
%%%%%%%%%%%%%%%%%%%%%%%%%%%%%%%%%%%%%%%%%%%%%%%%%%%%%%%%%%%%%%%%%%%%%%%%%%%%%%%%
%
\subsection{Direct limit manifolds}
  \label{SEC_direct_limit_manifolds}
This section will be devoted to establishing the notion of a direct  
  limit manifold, as well as some of its basic properties.
We also take the opportunity to introduce some of the notions we shall
  frequently use.

Here and throughout, a topological manifold \(X\) of dimension \(n\)  is a
  second countable Hausdorff space, where every point \(x\in X\) has an open
  neighborhood homeomorphic to an open subset in \({\mathbb R}^n\).  
A chart around \(x\in X\) is a triple \(({\mathcal O},\phi,{\mathcal V})\)
  where ${\mathcal V}$ is an open subset
  of ${\mathbb R}^n$, ${\mathcal O}$ is an open neighborhood of \(x\), and
  $\phi:{\mathcal O}\rightarrow {\mathcal V}$ is a homeomorphism.
An atlas for the topological manifold \(X\) consists of a family of charts
  ${({\mathcal O}_i,\phi_i,{\mathcal V}_i)}_{i\in I}$ such that
  $X=\cup_{i\in I} {\mathcal O}_i$. We call the collection a smooth atlas
  provided all transition maps are  $C^\infty$.

In the category \(\text{Man}^n\) of \(n\)-dimensional smooth manifolds we
  will have to consider, during later constructions, direct limits
  of so-called \emph{embedding diagrams}
  \begin{align*}                                                          %% EQN
    (\mathbf{M},{\bm \psi}):\
    M_1\xrightarrow{\psi_1}M_2\xrightarrow{\psi_2} M_3\rightarrow\cdots,
    \end{align*}
  where all arrows are smooth embeddings between smooth manifolds. 
From this we obtain a directed system by defining 
  \begin{align*}                                                          %% EQN
    &\psi^k_j:M_j\rightarrow M_k
    \\
    &\psi^k_j=\psi_{k-1}\circ\cdots\circ \psi_j.
    \end{align*}
The direct limit of such a diagram is by definition a tuple
  \((\overline{M},\{\iota_k\}_{k\in {\mathbb N}})\), where
  \(\overline{M}\) is a smooth \(n\)-dimensional manifold and for every
  \(k\in {\mathbb N}\) the map \(\iota_k:M_k\rightarrow \overline{M}\) is
  smooth so that we  have the following universal property.
\begin{itemize}
  \item Given smooth maps \(f_k:M_k\rightarrow X\) such that
    \(f_{k+1}\circ \psi_k=f_k\) there exists a uniquely determined smooth
    map \(f:\overline{M}\rightarrow X\) such that \(f\circ \iota_k=f_k\).
  \end{itemize}
An immediate consequence of the universal property is that
  \((\overline{M},\{\iota_k\}_{k\in {\mathbb N}})\) is unique up to natural
  diffeomorphism.
It is also an easy exercise to show that the \(\iota_k\) have to be smooth
  embeddings.
This means any other realization of the direct limit
  \((\overline{M}',\{\iota_k'\}_{k\in {\mathbb N}})\)
  is related to the first by a uniquely determined diffeomorphism
  \(\phi:\overline{M}\rightarrow \overline{M}'\) satisfying \(\phi\circ
  \iota_k=\iota_k'\).
The construction of the direct limit is well-known, but we briefly recall
  that it can be defined as
  \begin{align*}                                                          %% EQN
    \overline{M}:=\Big( \bigcup_{k\in \mathbb{N}} M_k\times \{k\}\Big)
    /\sim
    \end{align*}
  where two points are equivalent if there exists \(\psi_j^k\) mapping one
  to the other.
Here are some standard properties of the direct limit.
\begin{enumerate}[(DL-1)]                                                 %% NUM
  \item 
  The smooth maps $f:\overline{M}\rightarrow X$ are in one-to-one
  correspondence to families of smooth maps $\{f_k\}_{k\in {\mathbb N}}$,
  where $f_k:M_k\rightarrow X$ such that $f_{k+1}\circ\psi_k=f_k$.
  \item 
  Each tensor \(\overline{T}\) on \(\overline{M}\) uniquely corresponds to
  a family of tensors \(\{T_k\}_{k\in {\mathbb N}}\)  on the
  \(\{M_k\}_{k\in \mathbb{N}}\) such that \(T_k\) is
  a tensor on \(M_k\) and \(\psi_k^\ast T_{k+1}=T_k\), and each such
  family gives rise to a unique tensor on \(\overline{M}\).
  \end{enumerate}

Given an embedding diagram \((\mathbf{V}, {\bm \phi})\), an inclusion,
  denoted \(\mathbf{\Gamma}:(\mathbf{M},{\bm
  \psi})\rightarrow (\mathbf{V},{\bm \phi})\), consists of a family of
  smooth maps \(a_k:M_k\rightarrow V_k\) such that the following diagrams
  are commutative:
  \begin{align*}                                                          %% EQN
    \begin{CD}
      M_k @> \psi_k>>  M_{k+1}\\
      @V a_k VV @V a_{k+1} VV\\
      V_k@>\phi_k >>   V_{k+1}
        \end{CD}
      \end{align*}
It follows immediately that the \(a_k\) must be embeddings, and hence we obtain an
  embedding between the direct limits.
Let us call \(\mathbf{\Gamma}\) \emph{surjective} provided that for every
  \(v_k\in V_k\) there exists \(\ell\geq k\) and \(m_\ell\in M_\ell\) such
  that \(a_\ell(m_\ell)=\phi^\ell_k(v_k)\).
In this case the induced maps between the limits are diffeomorphisms.

As an example, consider an embedding diagram \((\mathbf{M}, {\bm \psi})\)
  and a strictly monotonic map \(\sigma:{\mathbb N}\rightarrow {\mathbb
  N}\).
Construct another embedding diagram \((\mathbf{V}, {\bm \phi})\), where
  \(V_k := M_{\sigma(k)}\) and \(\phi_k:V_k \to V_{k+1}\) is given by
  \(\phi_k =\psi_{\sigma(k)}^{\sigma(k+1)}\).
On one hand, we might regard \((\mathbf{V}, {\bm \phi})\) as a
  subsequence or subsystem of \((\mathbf{M}, {\bm \psi})\), however we will
  instead regard \((\mathbf{M}, {\bm \psi})\) as a surjective inclusion of
  \((\mathbf{V}, {\bm \phi})\) with the family of maps \(a_k:M_k \to V_k\)
  given by  \(a_k= \psi_k^{\sigma(k)}\).
Then, because this diagram commutes:
  \begin{align*}                                                            %% EQN
    \begin{CD}
      M_1 @>\psi_1>>  M_2 @>\psi_2>>  M_3@>\psi_3>> \cdots\\
      @V \psi_1^{\sigma(1)} VV  @V \psi_2^{\sigma(2)}  VV   @V \psi_3^{\sigma(3)} VV\\
      M_{\sigma(1)}  @>\psi_{\sigma(1)}^{\sigma(2)} >>  M_{\sigma(2)}
      @>\psi_{\sigma(2)}^{\sigma(3)} >> M_{\sigma(3)}
      @>\psi_{\sigma(3)}^{\sigma(4)}>>\cdots
      \end{CD}
    \end{align*}
it trivially follows that the following diagram commutes:
  \begin{align*}                                                            %% EQN
    \begin{CD}
      M_1 @>\psi_1>>  M_2 @>\psi_2>>  M_3@>\psi_3>> \cdots\\
      @V a_1 VV  @V a_2 VV   @V a_3 VV\\
      V_{1}  @>\phi_{1} >>  V_{2}
      @>\phi_{2} >> V_{3}
      @>\phi_{3}>>\cdots
      \end{CD}
    \end{align*}
From this it immediately follows that if \((\overline{M},
  \{\iota_k\}_{k\in \mathbb{N}})\) is the direct limit of \((\mathbf{M}, {\bm
  \psi})\), then \((\overline{M}, \{\iota_{\sigma(k)}\}_{k\in \mathbb{N}})\)
  is the direct limit of \((\mathbf{V}, {\bm \phi})\), and the limit
  manifolds are diffeomorphic via the identity map.
Or, put another way, regarding \((\mathbf{V}, {\bm \phi})\) as a
  subsequence of \((\mathbf{M}, {\bm \psi})\), we see that each embedding
  diagram has the same direct limit.

We note that the above claims can be proved by reducing them to the
  following trivial result, the proof of which is left to the reader.

%%%%%%%%%%%%%%%%%%%%%%%%%%%%%%%%%%%%%%%%%%%%%%%%%%%%%%%%%%%%%%%%%%%%%%%%%%%%%%%%
%%%%%%%%%%                            LEMMA                            %%%%%%%%%
%%%%%                                                                       %%%%
\begin{lemma}[exhausting subsets and direct limit manifolds]\hfill \\
  \label{LEM_exhausting_subsets}
Let \(\{M_k\}_{k\in \mathbb{N}}\) be a sequence of sets such that 
  \begin{enumerate}                                                       %% NUM
    \item 
    Each \(M_k\) carries the structure of a smooth manifold.
    \item 
    As sets,  each \(M_k\) is a subset of \(M_{k+1}\), and  \(M_k\subset
    M_{k+1}\) is an open subset in the \(M_{k+1}\) topology.
    \item 
    The smooth structure on \(M_k\) equals the smooth structure induced
    from \(M_{k+1}\).
    \end{enumerate}
Then the union,  \(\overline{M}:=\cup_{k=1}^{\infty} M_k\) carries a natural
  second countable Hausdorff topology and a uniquely determined smooth
  manifold structure, such that the natural inclusion \(M_k\rightarrow
  \overline{M}\) for every \(k\in {\mathbb N}\) is a smooth embedding as
  an open subset.
\end{lemma}
%%%%%                                                                       %%%%
%%%%%%%%%%                                                             %%%%%%%%%
%%%%%%%%%%%%%%%%%%%%%%%%%%%%%%%%%%%%%%%%%%%%%%%%%%%%%%%%%%%%%%%%%%%%%%%%%%%%%%%%
%

We conclude Section \ref{SEC_direct_limit_manifolds} by noting that in
  later sections we make use of direct limit manifolds in two ways: As
  targets for our pseudoholomorphic curves, and as domains of our curves.
In the target case, it is sufficient for our purposes to use the
  exhausting subset perspective, and here we will use the notation \(W_k
  \subset W_{k+1}\), etc.
However, in the domain case we will definitely rely on the language and
  results of embedding diagrams and direct limit manifolds.
Roughly speaking, we will employ a version of Gromov compactness with free
  boundary which, as a limit, yields a pseudoholomorphic curve with domain
  which is a compact manifold  with boundary.
By iterating this procedure we get a sequence of domains \(S_k\) which
  are larger and larger in the sense that there exists holomorphic
  embeddings \(\psi_k:S_k\to S_{k+1}\).
Thus in order to construct the desired limit curve, we need the sequence
  \((S_k, \psi_k)\) to give rise to an embedding diagram, and hence a direct
  limit manifold.

%%%%%%%%%%%%%%%%%%%%%%%%%%%%%%%%%%%%%%%%%%%%%%%%%%%%%%%%%%%%%%%%%%%%%%%%%%%%%%%%
%%%%%%%%%%                         SUB-SECTION                         %%%%%%%%%
%%%%%%%%%%%%%%%%%%%%%%%%%%%%%%%%%%%%%%%%%%%%%%%%%%%%%%%%%%%%%%%%%%%%%%%%%%%%%%%%
%
\subsection{Riemann surfaces}
  \label{SEC_riemann_surfaces} 
Here we aim to define a decorated marked nodal Riemann surface with
  boundary, as well as genus and arithmetic genus.
All of these notions will be utilized in later sections.

%%%%%%%%%%%%%%%%%%%%%%%%%%%%%%%%%%%%%%%%%%%%%%%%%%%%%%%%%%%%%%%%%%%%%%%%%%%%%%%%
%%%%%%%%%%                          DEFINITION                         %%%%%%%%%
%%%%%                                                                       %%%%
\begin{definition}[compact region]\hfill \\
  \label{DEF_compact_region}
Let \(W\) be a manifold.  
Suppose \(\mathcal{U}\subset W\) is an open set for which its closure
  \({\rm cl}(\mathcal{U})\) inherits from \(W\) the structure of a
  smooth compact manifold possibly with boundary.
Then we call \({\rm cl}(\mathcal{U})\) a \emph{compact region} in \(W\).
\end{definition}
%%%%%                                                                       %%%%
%%%%%%%%%%                                                             %%%%%%%%%
%%%%%%%%%%%%%%%%%%%%%%%%%%%%%%%%%%%%%%%%%%%%%%%%%%%%%%%%%%%%%%%%%%%%%%%%%%%%%%%%
%

%%%%%%%%%%%%%%%%%%%%%%%%%%%%%%%%%%%%%%%%%%%%%%%%%%%%%%%%%%%%%%%%%%%%%%%%%%%%%%%%
%%%%%%%%%%                          DEFINITION                         %%%%%%%%%
%%%%%                                                                       %%%%
\begin{definition}[almost complex structures]\hfill \\
  \label{DEF_almost_complex}
Let \(W\) be a smooth finite dimensional manifold equipped with a smooth
  section of the endomorphism bundle  \(J\in \Gamma({\rm End}(TW))\) over
  \(W\) for which \(J\circ J =- {\rm Id}\).
We call \(J\) an almost complex structure, and we call \((W, J)\) an
  almost complex manifold.
\end{definition}
%%%%%                                                                       %%%%
%%%%%%%%%%                                                             %%%%%%%%%
%%%%%%%%%%%%%%%%%%%%%%%%%%%%%%%%%%%%%%%%%%%%%%%%%%%%%%%%%%%%%%%%%%%%%%%%%%%%%%%%
%

%%%%%%%%%%%%%%%%%%%%%%%%%%%%%%%%%%%%%%%%%%%%%%%%%%%%%%%%%%%%%%%%%%%%%%%%%%%%%%%%
%%%%%%%%%%                          DEFINITION                         %%%%%%%%%
%%%%%                                                                       %%%%
\begin{definition}[nodal Riemann surface]\hfill\\
  \label{DEF_nodal_riemann_surface}
A \emph{nodal Riemann surface} is a triple \((S, j, D)\), where \(S\) is a
  real two-dimensional manifold, possibly with boundary, equipped with a
  smooth almost complex structure \(j\).
Furthermore, \(D\subset S\setminus \partial S\), is an unordered discrete
  closed set of pairs \(D= \{\overline{d}_1,\underline{d}_1,
  \overline{d}_2,\underline{d}_2,\ldots \} \) which we call nodal points,
  and the pairs \((\underline{d}_i, \overline{d}_i)\) we call nodal pairs.
A \emph{marked nodal Riemann surface} is the four-tuple \((S, j, \mu, D)\)
  where \((S, j, D)\) is a nodal Riemann surface, and where \(\mu\subset
  S\setminus (D\cup \partial S) \) is a discrete closed set of points.
\end{definition}
%%%%%                                                                       %%%%
%%%%%%%%%%                                                             %%%%%%%%%
%%%%%%%%%%%%%%%%%%%%%%%%%%%%%%%%%%%%%%%%%%%%%%%%%%%%%%%%%%%%%%%%%%%%%%%%%%%%%%%%
%

It is worth noting that we allow for the possibility that either of
  \(\mu\) or \(D\) may be empty, finite, or infinite.
That we allow either to be infinite is both novel and necessary for our
  applications, however by requiring that each set is both closed and
  discrete, we avoid the complicated issues arising from the existence of
  accumulations points of marked/nodal points.
Additionally, we make the requirement that the set of special points
  \(\mu\cup D\) is disjoint from the boundary.
This is a non-standard condition, but relevant for our approach because
  our curves have \emph{free} boundary; that is, no boundary condition
  is specified, and it will by natural to remove a small neighborhood of the
  boundary.
As such, we preemptively guarantee that in doing so, we do not remove any
  special points.

%%%%%%%%%%%%%%%%%%%%%%%%%%%%%%%%%%%%%%%%%%%%%%%%%%%%%%%%%%%%%%%%%%%%%%%%%%%%%%%%
%%%%%%%%%%                          DEFINITION                         %%%%%%%%%
%%%%%                                                                       %%%%
\begin{definition}[Genus]\hfill \\
  \label{DEF_genus}
Let $S$ be a connected compact two-dimensional manifold with boundary.  
We define ${\rm Genus} (S)$ to be the genus of the surface obtained by
  capping off the boundary components of $S$ by disks.
If $S$ is disconnected but compact, then we define
  ${\rm Genus}(S):=\sum_{k=1}^n{\rm Genus}(S_k)$ where the $S_k$ are
  the connected components of $S$.
If $S$ is non compact (but with at most countably infinite connected
  components), we define ${\rm Genus}(S):=\lim_{k\to\infty} {\rm
  Genus}(S_k)$, where $S_1\subset S_2 \subset S_3\subset \cdots$ is
  an exhausting sequence of compact regions\footnote{Recall that the
  notion of a compact region was provided in Definition
  \ref{DEF_compact_region}} in $S$.
\end{definition}
%%%%%                                                                       %%%%
%%%%%%%%%%                                                             %%%%%%%%%
%%%%%%%%%%%%%%%%%%%%%%%%%%%%%%%%%%%%%%%%%%%%%%%%%%%%%%%%%%%%%%%%%%%%%%%%%%%%%%%%
%

It is possible that a highly skeptical reader may be concerned that the
  above definition of genus may not be well defined, since a priori it
  could be the case that when \(S\) is non-compact, the definition of
  \({\rm Genus}(S)\) may depend upon the choice of exhausting sequence
  \(\{S_k\}_{k\in \mathbb{N}}\).
The proof of independence is elementary, so we do not provide it, but we
  mention the key ideas.
First, observe that the desired independence follows from showing that
  whenever \(S'\) and \(S''\) are compact regions in \(S\) satisfying 
  \begin{align*}                                                          %% EQN
    S' \subset {\rm Int}(S'')\qquad\text{and}\qquad S'' \subset{\rm Int}
    (S),
    \end{align*}
  we must have 
  \begin{equation*}                                                       %% EQN
      {\rm Genus}(S')\leq {\rm Genus}(S'').
    \end{equation*}
This in turn essentially follows from the fact that handle attaching only
  increases genus if both ends of the handle lie in the same connected
  component (before that attachment).

Before moving on, we will also need the concept of the arithmetic genus
  which we provide in Definition \ref{DEF_arithmetic_genus} below, but first
  we recall some additional notions.
Associated to a nodal Riemann surface is the topological space \(|S|\)
  defined by identifying a nodal point with the other point in its nodal
  pair; in other words, the space \(S/ (\underline{d}_\nu \sim
  \overline{d}_\nu)\).
As in Section 4.4 of \cite{BEHWZ}, we define \(S^D\) to
  be the oriented blow-up of \(S\) at the points $D$, and we
  let $\overline{\Gamma}_\nu:=\big( T_{\overline{d}_i}(S)\setminus
  \{0\}\big)/\mathbb{R}_+^*\subset S^D$ and \(
  \underline{\Gamma}_\nu:=\big( T_{\underline{d}_\nu}(S)\setminus
  \{0\}\big)/ \mathbb{R}_+^*\subset S^D\) denote the newly created
  boundary circles over the \(d_\nu\).

%%%%%%%%%%%%%%%%%%%%%%%%%%%%%%%%%%%%%%%%%%%%%%%%%%%%%%%%%%%%%%%%%%%%%%%%%%%%%%%%
%%%%%%%%%%                          DEFINITION                         %%%%%%%%%
%%%%%                                                                       %%%%
\begin{definition}[decorated marked nodal Riemann surface]\hfill\\
  \label{DEF_decorated_nodal_riemann_surface}
A \emph{decorated marked nodal Riemann surface} is a tuple \((S, j, \mu,
  D, r)\) where \((S, j, \mu, D)\) is a marked nodal Riemann surface, and
  \(r\) is a set of orientation reversing orthogonal maps
  \(\bar{r}_\nu:\overline{\Gamma}_\nu\to \underline{\Gamma}_\nu\) and 
  \(\underline{r}_\nu:\underline{\Gamma}_\nu\to \overline{\Gamma}_\nu\),
  which we call \emph{decorations}; here by orthogonal orientation
  reversing, we mean that \(r_\nu(e^{i\theta}z) = e^{-i\theta}r_\nu(z)\)
  for each \(z\in \Gamma_\nu \).
We also define \(S^{D,r}\) to be the smooth surface obtained
  by gluing the components of \(S^D\) along the boundary circles
  \(\{\overline{\Gamma}_1,\underline{\Gamma}_1,
  \overline{\Gamma}_2,\underline{\Gamma}_2, \ldots  \}\) via the
  decorations \(\bar{r}_\nu\) and \(\underline{r}_\nu\).
We will let \(\Gamma_\nu\) denote the special circles
  \(\overline{\Gamma}_\nu=\underline{\Gamma}_\nu\subset S^{D,r}\).
\end{definition}
%%%%%                                                                       %%%%
%%%%%%%%%%                                                             %%%%%%%%%
%%%%%%%%%%%%%%%%%%%%%%%%%%%%%%%%%%%%%%%%%%%%%%%%%%%%%%%%%%%%%%%%%%%%%%%%%%%%%%%%
%

We are now prepared to state the definition of the arithmetic genus.

%%%%%%%%%%%%%%%%%%%%%%%%%%%%%%%%%%%%%%%%%%%%%%%%%%%%%%%%%%%%%%%%%%%%%%%%%%%%%%%%
%%%%%%%%%%                          DEFINITION                         %%%%%%%%%
%%%%%                                                                       %%%%
\begin{definition}[arithmetic genus] \hfill\\
  \label{DEF_arithmetic_genus}
Let \(\mathbf{S}= (S, j, \mu, D)\) be a marked nodal nodal
  Riemann surface.
As above, let \(S^D\) be the oriented blow-up of \(S\) at the points
  \(D\), and let \(S^{D,r}\) denote the surface obtained by gluing \(S^D\)
  together along pairs of circles associated to pairs of nodal points.
We define the arithmetic genus of \(\mathbf{S}\) to be the genus of
  \(S^{D,r}\).
That is,
\begin{equation*}                                                         %% EQN
  {\rm Genus}_{arith}(\mathbf{S})  = {\rm Genus}(S^{D, r}).
  \end{equation*}
\end{definition}
%%%%%                                                                       %%%%
%%%%%%%%%%                                                             %%%%%%%%%
%%%%%%%%%%%%%%%%%%%%%%%%%%%%%%%%%%%%%%%%%%%%%%%%%%%%%%%%%%%%%%%%%%%%%%%%%%%%%%%%
%

We note that it is more standard to define the arithmetic genus in terms
  of a formula involving the genera of connected components, number of
  marked points, number of nodal points, etc.
It will be convenient for later applications to have the above definition
  at our disposal, however it is equivalent to the more standard formulaic
  definition.
Indeed, we establish this in Appendix \ref{SEC_formulat_arithmetic_genus}.

%%%%%%%%%%%%%%%%%%%%%%%%%%%%%%%%%%%%%%%%%%%%%%%%%%%%%%%%%%%%%%%%%%%%%%%%%%%%%%%%
%%%%%%%%%%                         SUB-SECTION                         %%%%%%%%%
%%%%%%%%%%%%%%%%%%%%%%%%%%%%%%%%%%%%%%%%%%%%%%%%%%%%%%%%%%%%%%%%%%%%%%%%%%%%%%%%
%
\subsection{Pseudoholomorphic curves}
  \label{SEC_pseudoholomorphic}
Here we will provide the definition of a pseudoholomorphic curve and a few
  other related notions like stability, generally immersed maps, boundary
  immersed maps, area, etc.

%%%%%%%%%%%%%%%%%%%%%%%%%%%%%%%%%%%%%%%%%%%%%%%%%%%%%%%%%%%%%%%%%%%%%%%%%%%%%%%%
%%%%%%%%%%                          DEFINITION                         %%%%%%%%%
%%%%%                                                                       %%%%
\begin{definition}[almost Hermitian structures]\hfill \\
  \label{DEF_almost_hermitian}
Let \((W, J)\) be a smooth finite dimensional almost complex manifold, and
  let \(g\) be a Riemannian metric.
We say the pair \((J, g)\) is an almost Hermitian structure on \(W\)
  provided that \(J\) is an isometry for \(g\).
In such a case we call \((W, J, g)\) an almost Hermitian manifold.
\end{definition}
%%%%%                                                                       %%%%
%%%%%%%%%%                                                             %%%%%%%%%
%%%%%%%%%%%%%%%%%%%%%%%%%%%%%%%%%%%%%%%%%%%%%%%%%%%%%%%%%%%%%%%%%%%%%%%%%%%%%%%%
%

%%%%%%%%%%%%%%%%%%%%%%%%%%%%%%%%%%%%%%%%%%%%%%%%%%%%%%%%%%%%%%%%%%%%%%%%%%%%%%%%
%%%%%%%%%%                          DEFINITION                         %%%%%%%%%
%%%%%                                                                       %%%%
\begin{definition}[marked nodal pseudoholomorphic curve]\hfill\\
  \label{DEF_pseudholomorphic_curve}
A \emph{marked nodal pseudoholomorphic curve} is a tuple \(\mathbf{u}=(u,
  S, j, W, J, \mu, D)\) with entries as follows.  
The triple \((S, j, \mu, D)\) is a marked nodal Riemann surface, 
The pair \((W, J)\) is a smooth real \(2n\)-dimensional almost complex
  manifold, and \(u:S\to W\) is a smooth map for which $J\cdot Tu =
  Tu\cdot j$.
Finally, we require that \(u(\overline{d}_i) = u(\underline{d}_i)\) for
  all \(i\in \mathbb{N}\).
\end{definition}
%%%%%                                                                       %%%%
%%%%%%%%%%                                                             %%%%%%%%%
%%%%%%%%%%%%%%%%%%%%%%%%%%%%%%%%%%%%%%%%%%%%%%%%%%%%%%%%%%%%%%%%%%%%%%%%%%%%%%%%
%

Unless otherwise specified, we will allow \(S\) to be non-compact,
  to have smooth boundary, and to have unbounded topology (i.e. countably
  infinite connected components, boundary components, and genus).
We will say that a pseudoholomorphic curve $\mathbf{u}$ is \emph{compact}
  provided \(S\) has the structure of a compact manifold with smooth
  boundary, we will say $\mathbf{u}$ is \emph{closed} provided $S$
  has the structure of a compact manifold without boundary, and we will
  say \(\mathbf{u}\) is \emph{connected} provided that \(|S|\) is connected.

%%%%%%%%%%%%%%%%%%%%%%%%%%%%%%%%%%%%%%%%%%%%%%%%%%%%%%%%%%%%%%%%%%%%%%%%%%%%%%%%
%%%%%%%%%%                          DEFINITION                         %%%%%%%%%
%%%%%                                                                       %%%%
\begin{definition}[decorated pseudoholomorhpic curve]\hfill \\
  \label{DEF_decorated}
A \emph{decorated} marked nodal pseudoholomorphic curve \((\mathbf{u},
  r)\) is a pair for which \(\mathbf{u}= (u, S, j, W, J, \mu, D)\) is a
  marked nodal pseudoholomorphic curve and \((S, j, \mu, D, r)\) is a
  decorated marked nodal Riemann surface.
With \(S^{D,r}\) defined as above, we observe that the smooth map \(u:S\to
  W\) lifts to a continuous map \(u:S^{D,r}\to W\).
\end{definition}
%%%%%                                                                       %%%%
%%%%%%%%%%                                                             %%%%%%%%%
%%%%%%%%%%%%%%%%%%%%%%%%%%%%%%%%%%%%%%%%%%%%%%%%%%%%%%%%%%%%%%%%%%%%%%%%%%%%%%%%
%

%%%%%%%%%%%%%%%%%%%%%%%%%%%%%%%%%%%%%%%%%%%%%%%%%%%%%%%%%%%%%%%%%%%%%%%%%%%%%%%%
%%%%%%%%%%                          DEFINITION                         %%%%%%%%%
%%%%%                                                                       %%%%
\begin{definition}[generally immersed]\hfill \\
  \label{DEF_generally_immersed}
Let \(\mathbf{u} = (u, S, j, W, J, \mu, D)\) be a marked nodal
  pseudoholomorphic curve.
We shall say \(\mathbf{u}\), or \(u:S\to W\), is \emph{generally
  immersed} provided that the set of critical points of \(u:S\to W\)
  has no accumulation point.
\end{definition}
%%%%%                                                                       %%%%
%%%%%%%%%%                                                             %%%%%%%%%
%%%%%%%%%%%%%%%%%%%%%%%%%%%%%%%%%%%%%%%%%%%%%%%%%%%%%%%%%%%%%%%%%%%%%%%%%%%%%%%%
%

%%%%%%%%%%%%%%%%%%%%%%%%%%%%%%%%%%%%%%%%%%%%%%%%%%%%%%%%%%%%%%%%%%%%%%%%%%%%%%%%
%%%%%%%%%%                          DEFINITION                         %%%%%%%%%
%%%%%                                                                       %%%%
\begin{definition}[stable pseudoholomorphic curve]\hfill\\
  \label{DEF_stable_pseudoholomorphic_curve}
A compact marked nodal pseudoholomorphic curve \((u, S, j, W, \mu, D) \)
  is said to be \emph{stable} if and only if for each connected component
  \(\widetilde{S}\) of \(S\) at least one of the following is true:
  \begin{enumerate}                                                       %% NUM
    \item 
    The restricted map \(u\big|_{\widetilde{S}}:\widetilde{S}\to W\)
    is non-constant.
    \item 
    \(\chi(\widetilde{S}) - \#(\widetilde{S}\cap \mu) - \#(\widetilde{S}
    \cap D) < 0\).
    \end{enumerate}
A non-compact marked nodal pseudoholomorphic curve \((u, S, j, W, \mu, D)
  \) is said to be \emph{stable} if and only if there exists a sequence
  \(\{S_k\}_{k\in \mathbb{N}}\) of compact real two dimensional manifolds,
  possibly with smooth boundary, with the following properties.
\begin{enumerate}                                                         %% NUM
  \item 
  for each \(k\in \mathbb{N}\) we have \(S_k\subset {\rm Int}(S_{k+1})
  \subset S\)
  \item 
  \(\cup_{k=1}^\infty S_k = S\) and \((\cup_{k=1}^\infty \partial S_k)\cap
  (\mu\cup D)= \emptyset\)
  \item 
  for each \(k\in \mathbb{N}\) the pseudoholomorphic curve
  \begin{equation*}                                                       %% EQN
    (u, S_k, j, W, \mu\cap S_k, D\cap S_k)
    \end{equation*}
  is stable.
  \end{enumerate}
\end{definition}
%%%%%                                                                       %%%%
%%%%%%%%%%                                                             %%%%%%%%%
%%%%%%%%%%%%%%%%%%%%%%%%%%%%%%%%%%%%%%%%%%%%%%%%%%%%%%%%%%%%%%%%%%%%%%%%%%%%%%%%
%

Note that if \((u, S, j, W, \mu, D) \) is a compact marked nodal
  pseudoholomorphic curve,  and \(\widetilde{S}\) is a connected
  component for which \mbox{\(\chi(\widetilde{S}) - \#(\widetilde{S}\cap
  \mu) - \#(\widetilde{S} \cap D) < 0\)}, then there exists a unique
  complete finite area hyperbolic metric of constant curvature \(-1\)
  on \(S':=S\setminus(\mu\cup D)\) which is in the same conformal class
  as \(j\) and for which each connected component of \(\partial S\)
  is a geodesic; we denote this metric by \(h^{j,\mu\cup D}\).

%%%%%%%%%%%%%%%%%%%%%%%%%%%%%%%%%%%%%%%%%%%%%%%%%%%%%%%%%%%%%%%%%%%%%%%%%%%%%%%%
%%%%%%%%%%                          DEFINITION                         %%%%%%%%%
%%%%%                                                                       %%%%
\begin{definition}[boundary-immersed pseudoholomorphic curve]\hfill \\
  \label{DEF_boundary_immersed_pseudoholomorphic_curve}
A compact marked nodal pseudoholomorphic curve \((u, S, j, W, \mu, D) \)
  is said to be \emph{boundary-immersed} if and only if either \(\partial
  S=\emptyset\) or  else the restricted map
  \begin{equation*}                                                       %% EQN
    u\big|_{\partial S}:\partial S\to W  
    \end{equation*}
  is an immersion.
\end{definition}
%%%%%                                                                       %%%%
%%%%%%%%%%                                                             %%%%%%%%%
%%%%%%%%%%%%%%%%%%%%%%%%%%%%%%%%%%%%%%%%%%%%%%%%%%%%%%%%%%%%%%%%%%%%%%%%%%%%%%%%
%

%%%%%%%%%%%%%%%%%%%%%%%%%%%%%%%%%%%%%%%%%%%%%%%%%%%%%%%%%%%%%%%%%%%%%%%%%%%%%%%%
%%%%%%%%%%                            LEMMA                            %%%%%%%%%
%%%%%                                                                       %%%%
\begin{lemma}[A dichotomy]\hfill \\
  \label{LEM_dichotomy}
Let \mbox{\(\mathbf{u}=(u, S, j, W, J, \mu, D)\)} be a proper
  boundary-immersed marked nodal pseudoholomorphic curve mapping into the
  almost Hermitian manifold \((W, J, g)\) which has no boundary.
Then for each connected component \(\widetilde{S}\subset S\), the
  restricted map \(u\big|_{\widetilde{S}}:\widetilde{S}\to W\) is either
  a constant map or else it is generally immersed in the sense of
  Definition \ref{DEF_generally_immersed}.
\end{lemma}
%%%%%                                                                       %%%%
%%%%%%%%%%                                                             %%%%%%%%%
%%%%%%%%%%%%%%%%%%%%%%%%%%%%%%%%%%%%%%%%%%%%%%%%%%%%%%%%%%%%%%%%%%%%%%%%%%%%%%%%
%
\begin{proof}
For each connected component \(\widetilde{S}\) of \(S\) we note
  their are two possible cases: either the set of critical points of
  \(u\big|_{\widetilde{S}}: \widetilde{S}\to W\) has an accumulation point,
  or it does not.
Because \(\mathbf{u}\) is boundary-immersed, it follows that all
  accumulation points are interior points.
However, recall that any pseudoholomorphic map
  \(u:\widetilde{S}\to W\) with connected domain and with an interior
  accumulation point of critical points must be a constant map; for
  details, see Lemma 2.4.1 from \cite{MS}.
The result is then immediate.
\end{proof}%%%%%%%%%%%%%%%%%%%%%%%%%%%%%%%%%%%%%%%%%%%%%%%%            END PROOF

%%%%%%%%%%%%%%%%%%%%%%%%%%%%%%%%%%%%%%%%%%%%%%%%%%%%%%%%%%%%%%%%%%%%%%%%%%%%%%%%
%%%%%%%%%%                          DEFINITION                         %%%%%%%%%
%%%%%                                                                       %%%%
\begin{definition}[area of pseudoholomorhpic curves]\hfill \\
  \label{DEF_area}
Let \mbox{\(\mathbf{u}=(u, S, j, W, J, \mu, D)\)} be a proper
  boundary-immersed marked nodal pseudoholomorphic curve.
Let \(S_{\rm const}\subset S\) denote the union of connected components
  of \(S\) on which \(u\) is a constant map.
Then by Lemma \ref{LEM_dichotomy} it follows that the map \(u:S\setminus
  S_{\rm const}\to W\) is generally immersed in the sense of Definition
  \ref{DEF_generally_immersed}.
Consequently on \(S\setminus S_{\rm const}\) we  can define the
  following metric
  \begin{equation*}                                                       %% EQN
    {\rm dist}_{u^*g}(\zeta_0,\zeta_1):=\inf
    \Big\{{\textstyle \int_0^1} \langle
    \dot{\gamma}(t),\dot{\gamma}(t)\rangle_{u^*g}^{\frac{1}{2}}dt:
    \gamma\in \mathcal{C}^1\big([0,1],S\big)\text{ and }
    \gamma(i)=\zeta_i\Big\},
    \end{equation*}
  where our convention will be that if $\zeta_0$ and
  $\zeta_1$ lie in different connected components, then ${\rm
  dist}_{u^*g}(\zeta_0,\zeta_1):=\infty$.
Thus we may regard $(S\setminus S_{\rm const},{\rm dist}_{u^*g})$
  as a metric space, in which case it can be equipped with Hausdorff
  measures $\mathcal{H}^k$.
Note that if $\mathcal{O}\subset S\setminus S_{\rm const}$ is an open set
  on which $u$ is an immersion, then $\mathcal{H}^{2}(\mathcal{O})={\rm
  Area}_{u^*g}(\mathcal{O})$.
As such, our convention will be to simply define the area
  of an arbitrary open set \(\mathcal{U}\subset S\setminus S_{\rm
  const}\) to be \({\rm Area}_{u^*g}(\mathcal{U}):=
  \mathcal{H}^{2}(\mathcal{U})\).
Finally, for an arbitrary open set \(\mathcal{U}\subset S\) we define
  \begin{equation*}                                                       %% EQN
    {\rm Area}_{u^*g}(\mathcal{U}):=\mathcal{H}^{2}(\mathcal{U}\setminus
    S_{\rm const}).
    \end{equation*}
\end{definition}
%%%%%                                                                       %%%%
%%%%%%%%%%                                                             %%%%%%%%%
%%%%%%%%%%%%%%%%%%%%%%%%%%%%%%%%%%%%%%%%%%%%%%%%%%%%%%%%%%%%%%%%%%%%%%%%%%%%%%%%
%

%%%%%%%%%%%%%%%%%%%%%%%%%%%%%%%%%%%%%%%%%%%%%%%%%%%%%%%%%%%%%%%%%%%%%%%%%%%%%%%%
%%%%%%%%%%                         SUB-SECTION                         %%%%%%%%%
%%%%%%%%%%%%%%%%%%%%%%%%%%%%%%%%%%%%%%%%%%%%%%%%%%%%%%%%%%%%%%%%%%%%%%%%%%%%%%%%
%
\subsection{Convergence of pseudoholomorphic curves}
  \label{SEC_convergence_of_pseudoholomorphic_curves}
Here we provide a few notions of convergence of pseudoholomorphic curves.
We start with the well known definition of Gromov convergence adapted to
  the case of having free boundary.
We also recall the definition of robust \(\mathcal{K}\)-convergence in a
  Gromov sense, which is taken from \cite{F2}.
And finally, we provide the novel notion of convergence in an exhaustive
  Gromov sense, given in Definition \ref{DEF_exhaustive_gromov_convergence}.

We begin with the notion of Gromov convergence, which we adapt slightly to
  allow our curves to have free boundary.

%%%%%%%%%%%%%%%%%%%%%%%%%%%%%%%%%%%%%%%%%%%%%%%%%%%%%%%%%%%%%%%%%%%%%%%%%%%%%%%%
%%%%%%%%%%                          DEFINITION                         %%%%%%%%%
%%%%%                                                                       %%%%
\begin{definition}[Gromov convergence]\hfill \\
  \label{DEF_gromov_convergence}
A sequence \(\mathbf{u}_k = (u_k,S_k,j_k,W, J_k, \mu_k, D_k)\) of
  compact marked nodal stable boundary-immersed pseudoholomorphic curves
  is said to converge in a Gromov-sense to a compact marked nodal stable
  boundary-immersed pseudoholomorphic curve \(\mathbf{u}=(u,S,j,W,J,\mu,
  D)\) provided the following are true for all sufficiently large
  \(k\in \mathbb{N}\).
\begin{enumerate}                                                         %% NUM
  \item 
  \(J_k\to J$ in $C^\infty\).
  \item 
  There exist sets of marked points 
  \begin{equation*}                                                       %% EQN
    \mu_k'\subset S_k\setminus (\partial S_k\cup \mu_k\cup D_k) 
    \qquad\text{and}\qquad
    \mu'\subset S\setminus (\partial S \cup \mu \cup D)
    \end{equation*}
  with the property that \(\#\mu' = \# \mu_k'< \infty\), and with the
  property that for each connected component \(\widetilde{S}_k\) of
  \(S_k\) we have
  \begin{equation*}                                                       %% EQN
    \chi(\widetilde{S}_k) - \#\big(\widetilde{S}_k\cap (\mu_k\cup \mu_k'
    \cup D_k)\big) < 0
    \end{equation*}
  and for each connected component \(\widetilde{S}\) of \(S\) we have
  \begin{equation*}                                                       %% EQN
    \chi(\widetilde{S}) - \#\big(\widetilde{S}\cap (\mu\cup \mu' \cup
    D)\big) < 0.
    \end{equation*}
  \item 
  There exists a decoration \(r\) for \(\mathbf{u}\), a sequence of
  decorations \(r_k\) for the \(\mathbf{u}_k\), and sequences of
  diffeomorphisms \(\phi_k: S^{D,r}\to S_k^{D_k,r_k}\) such that the
  following hold
  \begin{enumerate}                                                       %% NUM
    \item 
    \(\phi_k(\mu) = \mu_k\)
    \item 
    \(\phi_k(\mu') = \mu_k'\)
    \item 
    for each \(i=1,\ldots,\delta\) the curve \(\phi_k(\Gamma_i)\) is
    a \(h^{j_k,\mu_k\cup \mu_k'\cup D_k}\)-geodesic in the punctured
    surface \(S_k':= S_k\setminus (\mu_k\cup \mu_k'\cup D_k)\).
    \end{enumerate}
  \item 
  \(\phi_k^* h^{j_k,\mu_k\cup\mu_k'\cup D_k} \to h^{j,\mu\cup\mu'\cup
    D}\) in \(C_{loc}^\infty\big(S^{D,r}\setminus (\mu\cup\mu'\cup_i
    \Gamma_i) \big)\); here we have abused notation by letting
    \(h^{j,\mu\cup\mu'\cup D}\) also denote its lift to \(S^{D,r}\).
  \item 
  \(\phi_k^*u_k \to u$ in $C^0(S^{D,r})\).
  \item 
  \(\phi_k^*u_k \to u\) in \(C_{loc}^\infty(S^{D,r}\setminus
    \cup_i\Gamma_i)\).
  \item 
  For each connected component \(\Lambda\) of \(\partial \overline{S}\),
  the \(\phi_k^* h^{j_k,\mu_k\cup\mu_k'\cup D_k}\)-length of \(\Lambda\)
  is uniformly bounded away from \(0\) and \(\infty\).
  \end{enumerate}
\end{definition}
%%%%%                                                                       %%%%
%%%%%%%%%%                                                             %%%%%%%%%
%%%%%%%%%%%%%%%%%%%%%%%%%%%%%%%%%%%%%%%%%%%%%%%%%%%%%%%%%%%%%%%%%%%%%%%%%%%%%%%%
%

In general, one should not expect a sequence of pseudoholomorphic curves
  with free boundary to converge in a Gromov sense -- even when the curves
  are compact with bounded area, genus, connected components, etc.
However, if we are allowed to ``trim away'' some portion of those curves
  outside some compact set \(\mathcal{K}\), then the area and topology
  bounds are indeed sufficient to extract a subsequence for the trimmed
  curves.
We make this statement precise in Theorem \ref{THM_target_local} below,
  which was proved in \cite{F2}, but first it necessitates the statement of
  Definition \ref{DEF_K_proper_and_convergence} (\(\mathcal{K}\)-proper
  sequence of pseudoholomorphic curves) and Definition
  \ref{DEF_robust_K_convergence_gromov} (robust $\mathcal{K}$-convergence in
  Gromov sense).

\begin{definition}[$\mathcal{K}$-proper sequence of pseudoholomorphic
  curves]\hfill \\
  \label{DEF_K_proper_and_convergence}
Consider a sequence of maps \(u_k:S_k\to W\) to and from manifolds
  which possibly have boundary and may be non-compact.
Let \(\mathcal{K}\subset{\rm Int}(W)\) be a compact set in the interior
  of \(W\).
We call \(\{u_k\}_{k\in \mathbb{N}}\) a \emph{robustly
  \(\mathcal{K}\)-proper} sequence provided there exists another compact set
  \(\widetilde{\mathcal{K}}\subset {\rm Int}(W)\) for which
  \(\mathcal{K}\subset{\rm Int}(\widetilde{\mathcal{K}})\) and if
  \(u_k^{-1}(\widetilde{\mathcal{K}})\setminus \partial S_k\) is compact
  for all \(k\in \mathbb{N}\).
Similarly a single map \(u:S\to W\) is robustly \(\mathcal{K}\)-proper
  provided the constant sequence \(u,u,u,\ldots\) is robustly
  \(\mathcal{K}\)-proper.
\end{definition}
%%%%%                                                                       %%%%
%%%%%%%%%%                                                             %%%%%%%%%
%%%%%%%%%%%%%%%%%%%%%%%%%%%%%%%%%%%%%%%%%%%%%%%%%%%%%%%%%%%%%%%%%%%%%%%%%%%%%%%%
%

%%%%%%%%%%%%%%%%%%%%%%%%%%%%%%%%%%%%%%%%%%%%%%%%%%%%%%%%%%%%%%%%%%%%%%%%%%%%%%%%
%%%%%%%%%%                            REMARK                           %%%%%%%%%
%%%%%                                                                       %%%%
\begin{remark}
Definition \ref{DEF_K_proper_and_convergence} above is essentially a
  restatement of Definition 2.3 from \cite{F2}.
More precisely, the two definitions are equivalent but stated slightly
  differently.
Indeed, where Definition \ref{DEF_K_proper_and_convergence} states:
  \begin{quote}
  \emph{``We call \(\{u_k\}_{k\in \mathbb{N}}\) a \emph{robustly
    \(\mathcal{K}\)-proper} sequence provided there exists another compact set
    \(\widetilde{\mathcal{K}}\subset {\rm Int}(W)\) for which
    \(\mathcal{K}\subset{\rm Int}(\widetilde{\mathcal{K}})\) and if
    \(u_k^{-1}(\widetilde{\mathcal{K}})\setminus \partial S_k\) is compact
    for all \(k\in \mathbb{N}\).''}
  \end{quote}
  Definition 2.3 states:
  \begin{quote}
  \emph{``We call \(\{u_k\}_{k\in \mathbb{N}}\) a \emph{robustly
    \(\mathcal{K}\)-proper} sequence provided there exists another compact set
    \(\widetilde{\mathcal{K}}\subset {\rm Int}(W)\) for which
    \(\mathcal{K}\subset{\rm Int}(\widetilde{\mathcal{K}})\) and if
    \(u_k^{-1}(\widehat{\mathcal{K}})\setminus \partial S_k\) is compact for
    every compact set
    \(\widehat{\mathcal{K}}\subset\widetilde{\mathcal{K}}\).''}
  \end{quote}
Because \(\widetilde{\mathcal{K}}\) is a compact set contained in
  \(\widetilde{\mathcal{K}}\), it follows that any sequence of \(u_k\) which
  satisfy Definition 2.3 will also satisfy Definition 2.17.
Now suppose that \(u_k\) is a sequence satisfying Definition 2.17, so that
  \(u_k^{-1}(\widetilde{\mathcal{K}})\setminus \partial S_k\) is compact.
Let \(\widehat{\mathcal{K}}\subset \widetilde{\mathcal{K}}\) be a
  compact set.
Because \(W\) is a manifold, \(W\) is metrizable, and every compact set in
  a metrizable space is closed, so that \(\widetilde{\mathcal{K}}\) is
  closed, and hence \(u_k^{-1}(\widetilde{\mathcal{K}})\setminus \partial
   S_k\) is closed in the subspace topology of \(S_k\setminus \partial
   S_k\).
However, \(u_k^{-1}(\widehat{\mathcal{K}})\setminus \partial S_k\subset
  u_k^{-1}(\widetilde{\mathcal{K}})\setminus \partial S_k\) with the subset
  closed and the superset compact, so that the subset is compact.
In other words, for each compact \(\widehat{\mathcal{K}}\subset
  \widetilde{\mathcal{K}}\) we have
  \(u_k^{-1}(\widehat{\mathcal{K}})\setminus \partial S_k\) is compact,
  and thus the sequence \(u_k\) satisfies Definition 2.3 whenever it
  satisfies Definition \ref{DEF_K_proper_and_convergence}.
We conclude that the two definitions are indeed equivalent.
\end{remark}
%%%%%                                                                       %%%%
%%%%%%%%%%                                                             %%%%%%%%%
%%%%%%%%%%%%%%%%%%%%%%%%%%%%%%%%%%%%%%%%%%%%%%%%%%%%%%%%%%%%%%%%%%%%%%%%%%%%%%%%
%

With this definition in hand, we can now provide the notion of robust
  \(\mathcal{K}\)-convergence in a Gromov sense.
We note that the following definition was originally provided in
  \cite{F2}, although we have slightly modified it here to allow for the
  possibility that the sequence of curves has marked and nodal points.

%%%%%%%%%%%%%%%%%%%%%%%%%%%%%%%%%%%%%%%%%%%%%%%%%%%%%%%%%%%%%%%%%%%%%%%%%%%%%%%%
%%%%%%%%%%                          DEFINITION                         %%%%%%%%%
%%%%%                                                                       %%%%
\begin{definition}[robust $\mathcal{K}$-convergence in Gromov sense]\hfill \\
  \label{DEF_robust_K_convergence_gromov}
Consider an almost Hermitian manifold given by \((W,J,g)\), a sequence
  of almost Hermitian structures \((J_k,g_k)\) for which \((J_k,g_k)\to
  (J,g)\) in \(\mathcal{C}_{loc}^\infty\), a compact set
  \(\mathcal{K}\subset {\rm Int}(W)\), and a robustly
  \(\mathcal{K}\)-proper sequence of marked nodal pseudoholomorphic curves
  \(\mathbf{u}_k=(u_k,S_k,j_k, W, J_k, \mu_k, D_k)\).
We say that the \(\mathbf{u}_k\) \emph{robustly}
  \(\mathcal{K}\)\emph{-converge in a Gromov sense} provided there exists
  a compact set \(\widetilde{\mathcal{K}}\subset {\rm Int}(W)\) for which
  \(\mathcal{K}\subset {\rm Int}(\widetilde{\mathcal{K}})\), and there
  exist compact regions \(\tilde{S}_k\subset S_k\) with the property that
  \(u_k(S_k\setminus \tilde{S}_k)\subset W\setminus
  \widetilde{\mathcal{K}}\) for all \(k\in \mathbb{N}\), \((\mu_k \cup
  D_k)\cap \partial \tilde{S}=\emptyset\) for all \(k\in \mathbb{N}\), and
  the domain restricted pseudoholomorphic curves
  \begin{equation*}                                                       %% EQN
    \tilde{\mathbf{u}}_k:=(u_k,\tilde{S}_k,j_k,W, J_k, \mu_k \cap
    \tilde{S}_k, D_k\cap \tilde{S}_k)
    \end{equation*}
  are stable in the sense of Definition
  \ref{DEF_stable_pseudoholomorphic_curve} and converge in a Gromov sense
  to a stable compact marked nodal
  boundary-immersed pseudoholomorphic curve
  \begin{equation*}                                                       %% EQN
    \mathbf{u}:=(u, S, j, W, J, \mu, D).
    \end{equation*}
We additionally require that the sequence of marked points \(\mu_k'\)
  added to the \((\tilde{S},j_k)\) to obtain Gromov convergence are chosen
  so that lengths of each connected component of \(\partial \tilde{S}_k\),
  computed with respect to the associated Poincar\'{e} metric \(h^{j_k,
  \mu_k\cup \mu_k'\cup D_k}\), are uniformly bounded away from zero
  and infinity.
\end{definition}
%%%%%                                                                       %%%%
%%%%%%%%%%                                                             %%%%%%%%%
%%%%%%%%%%%%%%%%%%%%%%%%%%%%%%%%%%%%%%%%%%%%%%%%%%%%%%%%%%%%%%%%%%%%%%%%%%%%%%%%
%

Before providing the definition of convergence in an exhaustive Gromov
  sense, we will need the following notion, which is a special case of the
  embedding diagrams of Section \ref{SEC_direct_limit_manifolds}.

%%%%%%%%%%%%%%%%%%%%%%%%%%%%%%%%%%%%%%%%%%%%%%%%%%%%%%%%%%%%%%%%%%%%%%%%%%%%%%%%
%%%%%%%%%%                          DEFINITION                         %%%%%%%%%
%%%%%                                                                       %%%%
\begin{definition}[properly exhausting regions]\hfill\\
  \label{DEF_properly_exhausting_regions}
Let \((\overline{W}, \overline{J}, \bar{g})\) be an almost Hermitian
  manifold, which need not be compact.
We say a sequence of almost Hermitian manifolds \((W_k, J_k, g_k)\)
  \emph{properly exhaust} \((\overline{W}, \overline{J}, \bar{g})\)
  provided the following hold.
  \begin{enumerate}                                                       %% NUM
    \item 
    For each \(k\in \mathbb{N}\) we have \(W_k\subset W_{k+1}\), and
    moreover \(W_k\) is an open subset of \(W_{k+1}\) in the \(W_{k+1}\)
    topology.
    \item 
    \(\overline{W} = \bigcup_{k\in\mathbb{N}} W_k\)
    \item 
    The smooth structure on \(W_k\) equals the smooth structure induced
    from \(W_{k+1}\).
    \item 
    The set \({\rm cl}(W_k) \subset W_{k+1}\) is a compact manifold with
    smooth boundary.
    \item 
    Regarding \((J_k, g_k)\) as almost Hermitian structures on
    \(\overline{W}\), we require \((J_k, g_k)\to (\overline{J}, \bar{g})\)
    in \(\mathcal{C}_{loc}^\infty\).
    \end{enumerate}
\end{definition}
%%%%%                                                                       %%%%
%%%%%%%%%%                                                             %%%%%%%%%
%%%%%%%%%%%%%%%%%%%%%%%%%%%%%%%%%%%%%%%%%%%%%%%%%%%%%%%%%%%%%%%%%%%%%%%%%%%%%%%%
%

We are now prepared to state the novel definition of convergence of
  pseudoholomorphic curves in an exhaustive Gromov sense.

%%%%%%%%%%%%%%%%%%%%%%%%%%%%%%%%%%%%%%%%%%%%%%%%%%%%%%%%%%%%%%%%%%%%%%%%%%%%%%%%
%%%%%%%%%%                          DEFINITION                         %%%%%%%%%
%%%%%                                                                       %%%%
\begin{definition}[convergence in an exhaustive Gromov sense]\hfill \\
  \label{DEF_exhaustive_gromov_convergence}
Let \((\overline{W}, \overline{J}, \bar{g})\) be a smooth almost Hermitian
  manifold, not necessarily compact, and let \((W_k, J_k, g_k)\) be a
  sequence which properly exhausts \((\overline{W}, \overline{J},
  \bar{g})\), in the sense of Definition
  \ref{DEF_properly_exhausting_regions}.
Suppose further that the tuples \(\bar{\mathbf{u}}=(\bar{u},
  \overline{S}, \bar{j}, \overline{W}, \overline{J}, \bar{\mu},
  \overline{D})\) and, for each \(k\in \mathbb{N}\), \(\mathbf{u}_k=(u_k,
  S_k, j_k, W_k, J_k, \mu_k, D_k)\), are each marked nodal proper stable
  pseudoholomorphic curves without boundary.
We say the sequence \(\{\mathbf{u}_k\}_{k\in \mathbb{N}}\)
  converges to \(\bar{\mathbf{u}}\) in an \emph{exhaustive Gromov
  sense} provided there exists a collection of compact smooth two
  dimensional manifolds with boundary \(\{\overline{S}^\ell\}_{\ell \in
  \mathbb{N}}\) with \(\overline{S}^\ell\subset \overline{S}\) for each
  \(\ell\in \mathbb{N}\), and there exists a collection of compact smooth
  two dimensional manifolds with boundary \(\{S_k^\ell\}_{\substack{\ell
  \in \mathbb{N}\\ k\geq \ell}}\) with \(S_k^\ell \subset S_k\) for
  all \(k, \ell\in \mathbb{N}\) with \(k\geq \ell\) for which the
  following hold.
\begin{enumerate}                                                         %% NUM
  \item 
  \(\overline{S}^\ell\subset \overline{S}^{\ell+1}\setminus \partial
  \overline{S}^{\ell+1}\) for all \(\ell\in \mathbb{N}\)
  \item 
  \(\overline{S} = \bigcup_{\ell \in \mathbb{N}} \overline{S}^\ell\) 
  \item
  for each fixed \(k\in \mathbb{N}\) and each \(0\leq \ell \leq k-1\) we
    have \(S_k^\ell\subset S_k^{\ell+1}\setminus \partial S_k^{\ell+1}\)
  \item
  for each \(k\geq  \ell\in \mathbb{N}\) we have
  \begin{equation*}                                                       %% EQN
    u_{k} ^{-1}(W_{\ell})\subset S_k^{\ell},
    \end{equation*}
  \item 
  for each fixed \(\ell\in \mathbb{N}\), the sequence 
    \begin{equation*}                                                     %% EQN
      \big\{
      \big( u_k,\; 
      S_k^{\ell},\; 
      j_k,\;  
      \overline{W}, \;
      J_k, \;
      S_k^{\ell}\cap\mu_k,\; 
      S_k^{\ell}\cap D_k\big)
      \big\}_{k \geq \ell}
      \end{equation*}
    is a sequence of compact marked nodal stable boundary-immersed
    pseudoholomorphic curves which converges in a Gromov sense to the
    proper marked nodal stable boundary-immersed pseudoholomorphic curve
    \begin{equation*}                                                       %% EQN
      \big(\bar{u},\; 
      \overline{S}^{\ell},\; 
      \bar{j},\;  
      \overline{W}, \;
      \overline{J}, \;
      \overline{S}^{\ell}\cap\bar{\mu},\; 
      \overline{S}^{\ell}\cap \overline{D}\big).
      \end{equation*}
  \end{enumerate}
\end{definition}
%%%%%                                                                       %%%%
%%%%%%%%%%                                                             %%%%%%%%%
%%%%%%%%%%%%%%%%%%%%%%%%%%%%%%%%%%%%%%%%%%%%%%%%%%%%%%%%%%%%%%%%%%%%%%%%%%%%%%%%
%

Before moving on to the next section, in which we prove exhaustive Gromov
  compactness, we first address two issues which establish that our notion
  of exhaustive convergence is well defined.
The first is that limits are unique, and the second is that subsequences
  converge to the same limit.
We handle both of these results with Lemma
  \ref{LEM_properties_of_exhaustive_gromov_limit} below.

%%%%%%%%%%%%%%%%%%%%%%%%%%%%%%%%%%%%%%%%%%%%%%%%%%%%%%%%%%%%%%%%%%%%%%%%%%%%%%%%
%%%%%%%%%%                            LEMMA                            %%%%%%%%%
%%%%%                                                                       %%%%
\begin{lemma}[properties of the exhaustive Gromov limit]\hfill\\
  \label{LEM_properties_of_exhaustive_gromov_limit}
Let \((\overline{W}, \overline{J}, \bar{g})\) be a smooth almost Hermitian
  manifold, not necessarily compact, and let \((W_k, J_k, g_k)\) be a
  sequence which properly exhausts \((\overline{W}, \overline{J},
  \bar{g})\).
Suppose that \(\mathbf{u}_k=(u_k, S_k, j_k, W_k, J_k, \mu_k,
  D_k)\) is a sequence of pseudoholomorphic curves which converges in an
  exhaustive sense to both of the limit curves \(\bar{\mathbf{u}}=(\bar{u},
  \overline{S}, \bar{j}, \overline{W}, \overline{J}, \bar{\mu},
  \overline{D})\) and \(\dot{\bar{\mathbf{u}}}=(\dot{\bar{u}},
  \dot{\overline{S}}, \dot{\bar{j}}, \overline{W}, \overline{J},
  \dot{\bar{\mu}},\dot{\overline{D}})\). 
\emph{Then} there exists a holomorphic diffeomorphism \(\varphi:
 (\overline{S},\overline{j},\overline{\mu},\overline{D})\to
 (\dot{\overline{S}}, \dot{\bar{j}}, \dot{\bar{\mu}},
 \dot{\overline{D}})\) satisfying \(v\circ \varphi = u\).
In addition, any subsequence of the \(\mathbf{u}_k\) also converges to
  \(\bar{\mathbf{u}}\).
\end{lemma}
%%%%%                                                                       %%%%
%%%%%%%%%%                                                             %%%%%%%%%
%%%%%%%%%%%%%%%%%%%%%%%%%%%%%%%%%%%%%%%%%%%%%%%%%%%%%%%%%%%%%%%%%%%%%%%%%%%%%%%%
%
\begin{proof}
The first conclusion follows immediately from the definition of exhaustive
  Gromov convergence, the standard Gromov convergence for compact
  pseudoholomorphic curves with smooth boundary, and that direct limits of
  embedding diagrams are unique up to diffeomorphisms preserving additional
  tensors (like almost complex structures).
The second conclusion follows from these results together with the fact
  that subsequences of embedding diagrams have the same direct limit.
See Section \ref{SEC_direct_limit_manifolds} for details. 
\end{proof}%%%%%%%%%%%%%%%%%%%%%%%%%%%%%%%%%%%%%%%%%%%%%%%%            END PROOF

With all these preliminaries established, we can now state the main result
  of this manuscript.
\setcounter{CounterSectionEGC}{\value{section}}
\setcounter{CounterTheoremEGC}{\value{theorem}}
%%%%%%%%%%%%%%%%%%%%%%%%%%%%%%%%%%%%%%%%%%%%%%%%%%%%%%%%%%%%%%%%%%%%%%%%%%%%%%%%
%%%%%%%%%%                           THEOREM                           %%%%%%%%%
%%%%%                                                                       %%%%
\begin{theorem}[exhaustive Gromov compactness]\hfill \\
  \label{THM_exhaustive_gromov_compactness}
Let \((\overline{W}, \overline{J}, \bar{g})\) be a smooth almost Hermitian
  manifold, not necessarily compact, and let \((W_k, J_k, g_k)\) be a
  sequence which properly exhausts \((\overline{W}, \overline{J},
  \bar{g})\), in the sense of Definition
  \ref{DEF_properly_exhausting_regions}.
Suppose further that the sequence denoted by 
\begin{align*}                                                            %% EQN
  \{\mathbf{u}_k\}_{k\in \mathbb{N}}=\{(u_k, S_k, j_k, W_k, J_k, \mu_k,
  D_k)\}_{k\in \mathbb{N}}
  \end{align*}
  is a sequence of proper stable marked nodal pseudoholomorphic curves
  without boundary for which there also exists a sequence of large
  constants \(C_k\) with the property that for each fixed \(k\in
  \mathbb{N}\) the following hold 
  \begin{enumerate}[(C1)]
    \item %\label{EN_C1}
    \( \displaystyle{\sup_{\ell \geq k}} \; {\rm Area}_{u_\ell^*g_\ell}
    (\widehat{S}_\ell^k)\leq C_k \)
    \item %\label{EN_C2}
    \(\displaystyle{\sup_{\ell \geq k}} \; {\rm
    Genus}(\widehat{S}_\ell^k)\leq C_k \)
    \item %\label{EN_C3}
    \(\displaystyle{\sup_{\ell \geq k}}\; \# \big((\mu_\ell \cup
    D_\ell)\cap \widehat{S}_\ell^k\big)\leq C_k \)
    \end{enumerate}
  where \(\widehat{S}_\ell^k:=u_\ell^{-1}(W_k)\). 
\emph{Then} a subsequence converges in an exhaustive Gromov sense to
  \((\bar{u}, \overline{S}, \bar{j}, \overline{W}, \overline{J}, \bar{\mu},
  \overline{D})\) which is a proper stable marked nodal pseudoholomorphic
  curve without boundary.
\end{theorem}
%%%%%                                                                       %%%%
%%%%%%%%%%                                                             %%%%%%%%%
%%%%%%%%%%%%%%%%%%%%%%%%%%%%%%%%%%%%%%%%%%%%%%%%%%%%%%%%%%%%%%%%%%%%%%%%%%%%%%%%
%

%%%%%%%%%%%%%%%%%%%%%%%%%%%%%%%%%%%%%%%%%%%%%%%%%%%%%%%%%%%%%%%%%%%%%%%%%%%%%%%%
%%%%%%%%%%                           SECTION                           %%%%%%%%%
%%%%%%%%%%%%%%%%%%%%%%%%%%%%%%%%%%%%%%%%%%%%%%%%%%%%%%%%%%%%%%%%%%%%%%%%%%%%%%%%
%
\section{Proof of exhaustive Gromov compactness}
  \label{SEC_proof_exhaustive_Gromov_compactness}
The main purpose of this section is to establish our main result, namely
  the validity of exhaustive Gromov compactness; see Theorem
  \ref{THM_exhaustive_gromov_compactness} below.
Before doing so, we first state the target-local Gromov compactness
  theorem from \cite{F2}, and generalize it to handle the case that the
  curves in question are marked, nodal, and have constant components.

%%%%%%%%%%%%%%%%%%%%%%%%%%%%%%%%%%%%%%%%%%%%%%%%%%%%%%%%%%%%%%%%%%%%%%%%%%%%%%%%
%%%%%%%%%%                           THEOREM                           %%%%%%%%%
%%%%%                                                                       %%%%
\begin{theorem}[Target-local Gromov compactness]\hfill \\
  \label{THM_target_local}
Let $(M,J,g)$ be an almost Hermitian manifold, and let $(J_k,g_k)$ be
  a sequence of almost Hermitian structures which converge in $C^\infty$
  to $(J,g)$.
Also let $\mathcal{K}\subset {\rm Int}(M)$ be a compact region, and let
  $\mathbf{u}_k$ be a sequence of generally immersed $J_k$-curves which
  are robustly $\mathcal{K}$-proper and satisfy
  \begin{enumerate}
    \item ${\rm Area}_{u_k^*g_k}(S_k)\leq C_A <\infty$
    \item ${\rm Genus}(S_k)\leq C_G <\infty$.
    \end{enumerate}
Then a subsequence robustly $\mathcal{K}$-converges in a Gromov sense.
\end{theorem}
%%%%%                                                                       %%%%
%%%%%%%%%%                                                             %%%%%%%%%
%%%%%%%%%%%%%%%%%%%%%%%%%%%%%%%%%%%%%%%%%%%%%%%%%%%%%%%%%%%%%%%%%%%%%%%%%%%%%%%%
%
\begin{proof}
This is nothing more than a restatement of Theorem 3.1 from \cite{F2}.
\end{proof}%%%%%%%%%%%%%%%%%%%%%%%%%%%%%%%%%%%%%%%%%%%%%%%%            END PROOF
%

%%%%%%%%%%%%%%%%%%%%%%%%%%%%%%%%%%%%%%%%%%%%%%%%%%%%%%%%%%%%%%%%%%%%%%%%%%%%%%%%
%%%%%%%%%%                          COROLLARY                           %%%%%%%%
%%%%%                                                                       %%%%
\begin{corollary}[Target-local Gromov compactness for marked nodal stable
  curves]
  \label{COR_target_local_gromov_compactness}
Let $(W,J,g)$ be an almost Hermitian manifold, possibly with boundary,
  and let $(J_k,g_k)$ be a sequence of almost Hermitian structures which
  converge in $C^\infty$ to $(J,g)$.
Also let $\mathcal{K}^-, \mathcal{K}^+\subset {\rm Int}(W)$ be compact
  regions, satisfying \(\mathcal{K}^-\subset {\rm Int}(\mathcal{K}^+)\), and
  let $\mathbf{u}_k=(u_k, S_k, j_k, W, J_k, \mu_k, D_k)$ be a sequence of
  stable compact marked nodal pseudoholomorphic curves satisfying
  \(u_k(\partial S_k) \cap \mathcal{K}^+ = \emptyset\) and suppose there
  exists a large positive constant \(C>0\) for which
  \begin{enumerate}
    \item 
    \({\rm Area}_{u_k^*g_k}(S_k)\leq C\),
    \item 
    \({\rm Genus}(S_k)\leq C\),
    \item
    \(\# \big(\mu_k \cup D_k\big) \leq C\)
    \end{enumerate}
Then, after passing to a subsequence (still denoted with subscripts
  \(k\)), there exist compact surfaces with boundary \(\check{S}_k\subset
  S_k\) with the following properties
\begin{enumerate}                                                         %% NUM
  \item the following are compact pseudoholomorphic curves
    \begin{equation*}                                                     %% EQN
      (u_k, \check{S}_k, j_k, \mu_k \cap \check{S}_k, D_k
      \cap \check{S}_k)
      \end{equation*}
  \item these domain-restricted converge in a Gromov sense to a compact
    stable marked nodal boundary immersed pseudoholomorphic curve.
  \item \(u_k(S_k\setminus \check{S}_k) \subset W\setminus \mathcal{K}^-\)
  \end{enumerate}
\end{corollary}
%%%%%                                                                       %%%%
%%%%%%%%%%                                                             %%%%%%%%%
%%%%%%%%%%%%%%%%%%%%%%%%%%%%%%%%%%%%%%%%%%%%%%%%%%%%%%%%%%%%%%%%%%%%%%%%%%%%%%%%
%

Before proceeding with the proof, we comment on its statement and
  conclusions.
To that end, we view the result in light of an example.
Consider the compact unit ball in \(\mathbb{R}^4\) contained inside a
  concentric compact ball of radius three.
Also consider a sequence of compact pseudoholomorphic curves (with
  boundary) with images in the larger ball, and which map the boundary of
  the domains to a small neighborhood of the boundary of the large ball.
Assuming no nodes, no marked points, no genus, but assuming area bounds,
  does a subsequence converge?
Without trimming the curves, the answer is immediately no;
  counter-examples are easy to find.
Thus Corollary \ref{COR_target_local_gromov_compactness} and Theorem
  \ref{THM_target_local} both guarantee that one can find a trimming which
  guarantees that a subsequence converges.
However, these results also guarantee that the portions of the domains
  that get trimmed away are not in the region of interest.
To elaborate on this point, let \(\mathcal{K}^-\)  be the compact unit
  ball in our example, and let \(\mathcal{K}^+\) be the compact ball of
  radius two.
The region of interest will be \(\mathcal{K}^-\).
The third conclusion of Corollary
  \ref{COR_target_local_gromov_compactness} is that the image of the
  \(S_k\setminus \check{S}_k\) by \(u_k\) (that is, the image of the
  portions of the curves that get trimmed away) live in \(W\setminus
  \mathcal{K}^-\); that is, these sets live outside the unit ball.
In practice, one will be interested in the portion of curves that has
  image in \(B_1(0) = \mathcal{K}^-\), so the hypotheses demand that our
  curves have boundary outside the even larger set \(B_2(0) =
  \mathcal{K}^+\), and the conclusions guarantee that we only trim away
  portions that have image in \(B_3(0)\setminus B_1(0)= W\setminus
  \mathcal{K}^-\), which keeps intact those portions in which we are
  interested.

Note: There is an odd case which the above result covers as well, which we
  describe at present.
Namely, it is possible that the curves in the initial sequence have images
  always contained in \(W\setminus \mathcal{K}^+\), in which case it could
  be the case that \(S_k \setminus \check{S}_k = \emptyset\).
That is, the entire curve is trimmed away, leaving nothing but empty
  curves.
We allow this as a possibility, and note that the above result is highly
  non-trivial because of the condition that \(u_k(S_k\setminus \check{S}_k)
  \subset W\setminus \mathcal{K}^-\).
Indeed, this latter conclusion guarantees that if \(u_k(S_k)\cap
  \mathcal{K}^-\neq \emptyset\) for all \(k\in \mathbb{N}\), then
  \(\check{S}_k\neq \emptyset\) for all \(k\in \mathbb{N}\), and hence the
  limit curve is non-empty (by definition of Gromov convergence).
The key idea is that one is interested in the portion of the curves that
  have image in \(\mathcal{K}^-\); if no part of the curves have image in
  that region, then there is no part of the curves in which we are
  interested, and trimming down to the empty set immediately achieves the
  desired result, since it is trivial to show a sequence of empty curves
  converges.
On the other hand, the condition that \(u_k(S_k\setminus \check{S}_k)
  \subset W\setminus \mathcal{K}^-\) guarantees that any portion of the
  curves in which we are interested are not trimmed away.

We now proceed with the proof.

\begin{proof}
As a preliminary step, we let \(S_k^{\rm const}\) denote the union of
  the connected components of \(S_k\) on which \(u_k\) is locally constant.
We then define the sequence of sets
  \begin{equation*}                                                       %% EQN
    \sigma_k:=u_k\big(\mu_k\cup D_k \cup S_k^{\rm const}\big).
    \end{equation*}
Note that these are finite sets with boundedly many elements, and hence
  we may pass to a subsequence, still denoted with subscripts \(k\),
  so that they converge in the following sense.
There exists a finite set \(\sigma\subset W\) with the property that for
  each compact set \(K\subset W\) and each \(\epsilon>0\) there exists
  a \(k'=k'(\epsilon, K)\in \mathbb{N}\) such that for all \(k\geq k'\)
  we have
  \begin{equation*}                                                       %% EQN
    \sigma_k \subset (W\setminus K) \cup \bigcup_{p\in \sigma}
    \mathcal{B}_\epsilon^{\bar{g}}(p)
    \end{equation*}
  where \(\mathcal{B}_\epsilon^{\bar{g}}(p)\) is an open
  \(\bar{g}\)-metric ball of radius \(\epsilon\) centered at the
  point \(p\).
We then fix compact regions \(\widehat{W}^-, \widehat{W}^+\subset W\), so
  that
  \begin{align*}                                                          %% EQN
    \mathcal{K}^- \subset {\rm Int}(\widehat{W}^-), \qquad\widehat{W}^-
    \subset {\rm Int}(\widehat{W}^+), \quad\text{and}\quad
    \widehat{W}^+\subset {\rm Int}(\mathcal{K}^+)
    \end{align*}
  and
  \begin{align*}                                                          %% EQN
      \sigma \cap \big(\widehat{W}^+ \setminus {\rm
      Int}(\widehat{W}^-)\big) = \emptyset.
      \end{align*}
This is possible essentially because \(\sigma\) is finite and due to
  properties of a compact region; see Definition
  \ref{DEF_compact_region}.
We then define 
  \begin{align*}                                                          %% EQN
    \widehat{S}_k:= u_k^{-1}\big({\rm Int}(\widehat{W}^+) \big)\setminus
    S_k^{\rm const}
    \end{align*}
  for each \(k\in \mathbb{N}\), as well as the pseudoholomorphic curves
  \begin{equation*}                                                      %% EQN
    \hat{\mathbf{u}}_k=(u_k, \widehat{S}_k, j_k, {\rm Int}(\widehat{W}^+),
    J_k, \emptyset, \emptyset ).
    \end{equation*}
With this established, we see by Lemma \ref{LEM_dichotomy} that each
  \(\hat{\mathbf{u}}_k\) is generally immersed, and hence Theorem
  \ref{THM_target_local} applies to this sequence.
Consequently, we pass to the subsequence (still denoted with subscripts
  \(k\)) which robustly \(\mathcal{K}\)-converges in a Gromov sense,
  with \(\mathcal{K}=\widehat{W}^-\).
As such, we let \(\widetilde{\mathcal{K}}\subset {\rm
  Int}(\widehat{W}^+)\subset W\) be
  the compact region for which \(\widehat{W}^-\subset {\rm
  Int}(\widetilde{\mathcal{K}})\), and let \(\widetilde{S}_k\subset
  \widehat{S}_k\) denote the compact regions guaranteed by Theorem
  \ref{THM_target_local} for which \(u(\widehat{S}_k\setminus
  \widetilde{S}_k)\subset W\setminus \widetilde{\mathcal{K}}\)
  so that the (sub)sequence of compact (marked nodal) boundary-immersed
  pseudoholomorphic curves
  \begin{equation}\label{EQ_some_curves_1}                                %% EQN
    \tilde{\mathbf{u}}_k=(u_k, \widetilde{S}_k, j_k, W, J_k, \emptyset,
    \emptyset)
    \end{equation}
  converges in a Gromov sense to a compact (marked nodal)
  boundary-immersed pseudoholomorphic curve.

At this point, we let \(\check{S}_k^{\rm const} :=S_k^{\rm const}
  \cap u_k^{-1}(\widehat{W}^+)\); or in words, we let
  \(\check{S}_k^{\rm const}\) denote the connected components
  of \(S_k\) on which \(u_k\) is locally constant and takes values in
  \(\widehat{W}^+\).
We then define \(\check{S}_k: = \widetilde{S}_k \cup \check{S}_k^{\rm
  const}\), \(\check{\mu}_k:=\mu_k \cap \check{S}_k\), and \(\check{D}_k:
  = D_k \cap \check{S}_k\) for each \(k\in \mathbb{N}\), and we consider
  the sequence of compact marked nodal boundary-immersed stable
  pseudoholomorphic curves
  \begin{equation}\label{EQ_some_curves_2}                                %% EQN
    (u_k, \check{S}_k, j_k, W, J_k, \check{\mu}_k, \check{D}_k).
    \end{equation}
It is important to note that for all sufficiently large \(k\), these
  curves are nodal; that is, that the \(\check{D}_k\) are indeed sets of
  nodal pairs in \(\check{S}_k\).
A priori, the concern is that by trimming the curves from \(S_k\)
  to \(\widehat{S}_k\cup \check{S}_k^{\rm const}\) and again from
  \(\widehat{S}_k\cup \check{S}_k^{\rm const}\) to \(\check{S}_k\),
  we may have ``trimmed away'' one but not both points in a nodal pair.
We note that this is not possible for the trimming from \(S_k\) to
  \(\widehat{S}_k\cup \check{S}_k^{\rm const}\) since \(\widehat{S}_k
  \cup \check{S}_k^{\rm const}= u_k^{-1}\big({\rm
  Int}(\widehat{W}^+)\big)\), and since nodal pairs \(\{\underline{d},
  \overline{d}\}\) satisfy \(u_k(\underline{d})=u_k(\overline{d})\).
It is also not possible to split a nodal pair by trimming from
  \(\widehat{S}_k\cup \check{S}_k^{\rm const}\) to \(\check{S}_k\) for
  sufficiently large \(k\) because
  \begin{equation*}                                                       %% EQN
    u_k\big(\widehat{S}_k\cup \check{S}_k^{\rm const} \setminus
    \check{S}_k\big)\subset {\rm Int}(\widehat{W}^+)\setminus
    \widehat{W}^-,
    \qquad
    u_k\big(\mu_k\cup D_k\cup S_k^{\rm const}\big)=\sigma_k\to \sigma,
    \end{equation*}
  and
  \begin{equation*}                                                       %% EQN
    \sigma\cap \big(\widehat{W}^+\setminus {\rm Int}(\widehat{W}^- )\big)=
    \emptyset.
    \end{equation*}

Returning our attention to the curves provided in equation
  (\ref{EQ_some_curves_2}), we note that a further subsequence of
  this sequence converges in a Gromov sense to a stable compact marked
  nodal boundary-immersed pseudoholomorphic curve.
Indeed, the Gromov convergence follows from a standard straightforward
  argument which we briefly sketch.
First, we add sequences of marked points \(\mu_k'\subset (\check{S}_k\cup
  \mu_k \cup D_k)\) to stabilize the underlying domains.
Note that only boundedly many points need to be added to each curve
  since the number of connected components of the \(S_k\) on which \(u_k\)
  is locally constant is uniformly bounded; this follows from the fact
  that curves given in the hypotheses of Corollary
  \ref{COR_target_local_gromov_compactness} are stable and the fact that
  \(\#\mu_k +\#D_k< C\).
This guarantees the existence of the associated finite area hyperbolic
  metrics \(h^{j_k, \mu_k\cup \mu_k'\cup D_k}\) on the punctured surfaces
  \(\check{S}_k\setminus (\mu_k\cup \mu_k'\cup D_k)\) where the boundaries
  are geodesics.
Standard bubbling analysis follows, in which one further adds marked
  points as needed so gradient bounds (associated to the hyperbolic metric)
  are obtained for these maps.
Note that because we guaranteed that \(\sigma\cap
  \big(\widehat{W}^+\setminus {\rm Int}(\widehat{W}^-)\big)=\emptyset\), and
  because the curves
  \begin{equation*}                                                       %% EQN
    \tilde{\mathbf{u}}_k=(u_k, \widetilde{S}_k, j_k, W, J_k, \emptyset,
    \emptyset)
    \end{equation*}
  converge in a Gromov sense to a boundary-immersed
  curve \(\tilde{\mathbf{u}}_\infty\) for
  which we have \(\tilde{u}_\infty(\partial
  \widetilde{S}_\infty)\subset {\rm Int}(\widehat{W}^+)\setminus
  \widehat{W}^-\), it follows that there exist annular neighborhoods
  of the \(\partial \check{S}_k\) with moduli uniformly bounded away from
  zero which are disjoint from both \(\check{\mu}_k\cup \check{D}_k\).
Furthermore because the \(\tilde{\mathbf{u}}_k\) converge in a Gromov
  sense, we have gradient bounds for the maps in these same annular
  neighborhoods of the boundary.
It then follows that the additional marked points \(\mu_k'\) can
  be chosen so that the lengths of the boundary components of the
  \(\check{S}_k\) are uniformly bounded away from zero and infinity.
To complete the sketch of Gromov convergence, we note that the desired
  reparameterizations are given by the Uniformization theorem; the
  \(\mathcal{C}_{\rm loc}^\infty\) convergence away from nodes is given
  by by elliptic regularity (gradient bounds imply \(\mathcal{C}^\infty\)
  bounds); and \(\mathcal{C}^0\) convergence across the nodes follows
  from an application of the Monotonicity lemma.

We have thus established that after passing to a subsequence, still
  denoted with subscripts \(k\),  the curves provided in equation
  (\ref{EQ_some_curves_2}) converge in a Gromov sense to a compact stable
  marked nodal boundary-immersed pseudoholomorphic curve, and hence all that
  remains to show is that
  \(u_k(S_k\setminus \check{S}_k) \subset W\setminus \mathcal{K}^-\). 
However, this follows immediately from the following facts:
\begin{align*}                                                            %% EQN
  \check{S}_k = \widetilde{S}_k\cup \check{S}_k^{\rm const}
  \end{align*}
\begin{align*}                                                            %% EQN
  \check{S}_k^{\rm const} = S_k^{\rm const}\cap u_k^{-1}(\widehat{W}^+)
  \end{align*}
\begin{align*}                                                            %% EQN
  u(\widehat{S}_k\setminus \widetilde{S}_k)\subset W\setminus
  \widetilde{\mathcal{K}}
  \end{align*}
\begin{align*}                                                            %% EQN
  \mathcal{K}^- \subset\widehat{W}^- \subset {\rm
  Int}(\widetilde{\mathcal{K}}) \subset \widehat{W}^+
  \end{align*}
This completes the proof of Corollary
  \ref{COR_target_local_gromov_compactness}.
\end{proof}%%%%%%%%%%%%%%%%%%%%%%%%%%%%%%%%%%%%%%%%%%%%%%%%            END PROOF

We now aim to prove the main result of this manuscript, namely Theorem
  \ref{THM_exhaustive_gromov_compactness}, which we first restate for the
  reader's convenience.

\setcounter{CurrentSection}{\value{section}}
\setcounter{CurrentTheorem}{\value{theorem}}
\setcounter{section}{\value{CounterSectionEGC}}
\setcounter{theorem}{\value{CounterTheoremEGC}}
%%%%%%%%%%%%%%%%%%%%%%%%%%%%%%%%%%%%%%%%%%%%%%%%%%%%%%%%%%%%%%%%%%%%%%%%%%%%%%%%
%%%%%%%%%%                           THEOREM                           %%%%%%%%%
%%%%%                                                                       %%%%
\begin{theorem}[exhaustive Gromov compactness]\hfill \\
Let \((\overline{W}, \overline{J}, \bar{g})\) be a smooth almost Hermitian
  manifold, not necessarily compact, and let \((W_k, J_k, g_k)\) be a
  sequence which properly exhausts \((\overline{W}, \overline{J},
  \bar{g})\), in the sense of Definition
  \ref{DEF_properly_exhausting_regions}.
Suppose further that the sequence denoted by 
\begin{align*}                                                            %% EQN
  \{\mathbf{u}_k\}_{k\in \mathbb{N}}=\{(u_k, S_k, j_k, W_k, J_k, \mu_k,
  D_k)\}_{k\in \mathbb{N}}
  \end{align*}
  is a sequence of proper stable marked nodal pseudoholomorphic curves
  without boundary for which there also exists a sequence of large
  constants \(C_k\) with the property that for each fixed \(k\in
  \mathbb{N}\) the following hold 
  \begin{enumerate}[(C1)]
    \item \label{EN_C1}
    \( \displaystyle{\sup_{\ell \geq k}} \; {\rm Area}_{u_\ell^*g_\ell}
    (\widehat{S}_\ell^k)\leq C_k \)
    \item \label{EN_C2}
    \(\displaystyle{\sup_{\ell \geq k}} \; {\rm
    Genus}(\widehat{S}_\ell^k)\leq C_k \)
    \item \label{EN_C3}
    \(\displaystyle{\sup_{\ell \geq k}}\; \# \big((\mu_\ell \cup
    D_\ell)\cap \widehat{S}_\ell^k\big)\leq C_k \)
    \end{enumerate}
  where \(\widehat{S}_\ell^k:=u_\ell^{-1}(W_k)\). 
\emph{Then} a subsequence converges in an exhaustive Gromov sense to
  \((\bar{u}, \overline{S}, \bar{j}, \overline{W}, \overline{J}, \bar{\mu},
  \overline{D})\) which is a proper stable marked nodal pseudoholomorphic
  curve without boundary.
\end{theorem}
%%%%%                                                                       %%%%
%%%%%%%%%%                                                             %%%%%%%%%
%%%%%%%%%%%%%%%%%%%%%%%%%%%%%%%%%%%%%%%%%%%%%%%%%%%%%%%%%%%%%%%%%%%%%%%%%%%%%%%%
%
\begin{proof}
\setcounter{section}{\value{CurrentSection}}
\setcounter{theorem}{\value{CurrentTheorem}}

We begin by choosing sequences of open sets \(\widetilde{W}_k^-,
  \widetilde{W}_k^+\subset \overline{W}\) with the property that each
  \({\rm cl} (\widetilde{W}_k^-)\) and \({\rm cl} (\widetilde{W}_k^+)\) 
  is a smooth compact manifold with boundary, and 
  \begin{align*}                                                          %% EQN
    {\rm cl}(W_k) \subset \widetilde{W}_{k+1}^- \subset {\rm
    cl}(\widetilde{W}_{k+1}^-) \subset \widetilde{W}_{k+1}^+ \subset {\rm
    cl}(\widetilde{W}_{k+1}^+) \subset W_{k+1},
    \end{align*}
  and \(u_\ell \pitchfork \partial ({\rm cl}(\widetilde{W}_k^+))\),
  and \(u_\ell(\mu_k \cup D_\ell)\cap \partial ({\rm cl}(\widetilde{W}_k^+))
  = \emptyset \) for all \(k,\ell\in \mathbb{N}\); this is possible by
  Sard's theorem and the fact that the \(\partial ({\rm
  cl}(W_k))\) are smooth manifolds.
We then define the compact manifolds with smooth boundary
  \begin{align*}                                                          %% EQN
    \widetilde{S}_\ell^k := u_\ell^{-1}({\rm cl}(\widetilde{W}_k^+))
    \subset \widetilde{S}_\ell
    \end{align*}
  and observe that for each fixed \(k\in \mathbb{N}\), and for
  all sufficiently large \(\ell\in \mathbb{N}\), the following sequence of
  tuples
  \begin{equation*}                                                       %% EQN
    \mathbf{u}_\ell^k:=\big(u_\ell, \widetilde{S}_\ell^k,
    j_\ell, W_k, J_\ell, \mu_\ell \cap
    \widetilde{S}_\ell^k, D_\ell\cap \widetilde{S}_\ell^k\big)
    \end{equation*}
  are compact stable marked nodal pseudoholomorphic curves, and for each
  fixed \(k\) they  have uniformly bounded area, genus, and number of
  marked and nodal points.
Moreover, we let \(\mathcal{K}^-={\rm cl}(W_1)\) and
  \(\mathcal{K}^+ = {\rm cl}(\widetilde{W}_2^-)\), and observe that by
  construction, we have \(\mathcal{K}^-\subset {\rm Int}(\mathcal{K}^+)\),
  and \(u_\ell(\partial \widetilde{S}_\ell^k) \cap \mathcal{K}^+ =
  \emptyset\) for all \(\ell\geq 2\).
Consequently, the sequence \(\{\mathbf{u}_\ell^2\}_{\ell \geq 2}\)
  satisfies the hypotheses of Corollary
  \ref{COR_target_local_gromov_compactness}.

As such, we apply Corollary \ref{COR_target_local_gromov_compactness} to
  the sequence \(\{\mathbf{u}_\ell^2\}_{\ell\geq 2}\)  with
  \(\mathcal{K}^-\) and \(\mathcal{K}^+\) as defined above, 
  and thus we obtain a subsequence we denote by
  \(\big\{\mathbf{u}_{\ell_i^1}^2\big\}_{i\in \mathbb{N}}\), and
  obtain compact manifolds with smooth boundary
  \(\overline{\Sigma}_{\ell_i^1}^1\subset \widetilde{S}_{\ell_i^1}^2\) so
  that the sequence of compact marked nodal
  boundary-immersed pseudoholomorphic curves given by
  \begin{equation*}                                                       %% EQN
    \big(u_{\ell_i^1}, \overline{\Sigma}_{\ell_i^1}^1, j_{\ell_i^1},
    W_2, J_{\ell_i^1}, \mu_{\ell_i^1}\cap
    \overline{\Sigma}_{\ell_i^1}^1, D_{\ell_i^1} \cap
    \overline{\Sigma}_{\ell_i^1}^1  \big)
    \end{equation*}
  converge in Gromov sense to the limit curve 
  \begin{equation*}                                                       %% EQN
    \big(\bar{u}^1, \overline{\Sigma}^1, \bar{j}^1, W_2,
    \overline{J}, \mu^1, D^1  \big).
    \end{equation*}
Recall from the properties of Gromov convergence (see Definition
  \ref{DEF_gromov_convergence}) that this gives rise to a sequence of
  diffeomorphisms
  \begin{equation*}                                                       %% EQN
    \phi_{\ell_i^1}^1:(\overline{\Sigma}^1)^{D^1,r^1} \to
    (\overline{\Sigma}_{\ell_i^1}^1)^{D_{\ell_i^1},r_{\ell_i^1}}
    \end{equation*}
  between the circle-blown up limit domain and the circle-blown up
  sequence domains.
These diffeomorphisms have the property that 
  \begin{equation*}                                                       %% EQN
    (\phi_{\ell_i^1}^1)^*j_{\ell_i^1}\to \bar{j}^1
    \end{equation*}
  in \( \mathcal{C}_{\rm loc}^{\infty} \) on the compliment of the special
  circles in \((\overline{\Sigma}^1)^{D^1,r^1} \).

Next we consider the (sub)sequence of pseudoholomorphic curves
  \(\{\mathbf{u}_{\ell_i^1}^3\}_{i\in \mathbb{N}}\), and apply
  Corollary \ref{COR_target_local_gromov_compactness} to this sequence with
  \(\mathcal{K}^-={\rm cl}(W_2)\) and \(\mathcal{K}^+={\rm
  cl}(\widetilde{W}_3^-)\) as the associated compact sets.
Thus there exists a further subsequence, denoted with subscripts
  \(\ell_i^2\), and compact manifolds with smooth boundary
  \(\overline{\Sigma}_{\ell_i^2}^2\subset \widetilde{S}_{\ell_i^2}^3\) so that
  the sequence of compact marked nodal boundary-immersed pseudoholomorphic
  curves given by
  \begin{equation*}                                                       %% EQN
    \big(u_{\ell_i^2}, \overline{\Sigma}_{\ell_i^2}^2, j_{\ell_i^2},
    W_3, J_{\ell_i^2}, \mu_{\ell_i^2}\cap
    \overline{\Sigma}_{\ell_i^2}^2, D_{\ell_i^2} \cap
    \overline{\Sigma}_{\ell_i^2}^2  \big)
    \end{equation*}
  converge in Gromov sense to the limit curve 
  \begin{equation*}                                                       %% EQN
    \big(\bar{u}^2, \overline{\Sigma}^2, \bar{j}^2, W_3,
    \overline{J}, \mu^2, D^2  \big)
    \end{equation*}

Again we have diffeomorphisms
  \begin{equation*}                                                       %% EQN
    \phi_{\ell_i^2}^2:(\overline{\Sigma}^2)^{D^2,r^2} \to
    (\overline{\Sigma}_{\ell_i^2}^2)^{D_{\ell_i^2},r_{\ell_i^2}}
    \end{equation*}
  and on the complement of the special circles we have \(\mathcal{C}_{\rm
  loc}^\infty\) convergence of the almost complex structures
  \begin{equation*}                                                       %% EQN
    (\phi_{\ell_i^2}^2)^*j_{\ell_i^2}\to \bar{j}^2.
    \end{equation*}

Defining \(\Sigma^1:={\rm Int}(\overline{\Sigma}^1)\) and \(\Sigma^2:={\rm
  Int}(\overline{\Sigma}^2)\), our goal at present then becomes to construct
  a holomorphic embedding
  \begin{equation*}                                                       %% EQN
    \check{\psi}_1 : \big( \Sigma^1, \bar{j}^1\big)\to \big(\Sigma^2,
    \bar{j}^2\big)
    \end{equation*}
  which sends marked points to marked points and nodal points to nodal
  points, and for which \(\bar{u}^1 = \bar{u}^2\circ \check{\psi}_1\).
We accomplish this by considering the sequence of maps \(
  (\phi_{\ell_i^2}^2)^{-1} \circ\phi_{\ell_i^2}^1 :\overline{\Sigma}^1\to
  \overline{\Sigma}^2 \) which by definition are holomorphic with respect to
  domain (almost) complex structure \((\phi_{\ell_i^2}^1)^*j_{\ell_i^2}\)
  and target (almost) complex structure
  \((\phi_{\ell_i^2}^2)^*j_{\ell_i^2}\).

Because the domain and target (almost) complex structures converge
  smoothly away from the special circles, we can regard this as a sequence
  of pseudoholomorphic maps.
We claim the maps must have uniformly bounded gradient on the interior
  away from special circles.
Or in other words, we claim that if the gradient blows up along a sequence
  of points, those points must converge to either a special circle or the
  boundary \(\partial (\overline{\Sigma}^1)^{D^1, r^1}\).
Indeed, if this were not true, one could then employ bubbling analysis to
  construct a non-constant holomorphic map from \(\mathbb{C}\) into the open
  disk in \(\mathbb{C}\) which is impossible.

With interior gradient bounds established, we then note that by elliptic
  regularity, we have uniformly bounded derivatives on the interior away
  from special circles, and a subsequence then converges to the holomorphic
  embedding \(\check{\psi}_1:(\Sigma^1\setminus D^1,\bar{j}^1)\to (\Sigma^2,
  \bar{j}^2)\).
An application of the removable singularity theorem then extends this to
  the desired holomorphic embedding \(\check{\psi}_1:(\Sigma^1,\bar{j}^1)\to
  (\Sigma^2, \bar{j}^2)\).

We now claim that \(\bar{u}^2\circ \check{\psi}_1 = \bar{u}^1\).
Indeed, this follows essentially from the smooth interior convergence of
  the maps \( (\phi_{\ell_i^2}^2)^{-1} \circ\phi_{\ell_i^2}^1\to
  \check{\psi}_1\) together with the fact that we have \(\mathcal{C}_{\rm
  loc}^\infty\) convergence of the maps \(\bar{u}_{\ell_i^2}\circ
  \phi_{\ell_i^2}^2\to \bar{u}^2\).

At this point we collect our results, and we shall see that that the proof
  is nearly complete.
As a first step, we remove a thin annular open neighborhood of the
  boundary of \(\overline{\Sigma}^1\) to obtain the compact manifold with
  smooth boundary \(\check{\Sigma}^1\subset \overline{\Sigma}^1\); we
  similarly define \(\check{\Sigma}^2\subset \overline{\Sigma}^2\).
We also note that we have smooth (almost) complex manifolds without
  boundary given by \((\Sigma^1, \bar{j}^1)\) and \((\Sigma^2, \bar{j}^2)\)
  and a holomorphic map \(\check{\psi}_1: \Sigma^1\to \Sigma^2\) between
  them which also sends marked points to marked points and nodal points to
  nodal points, and which satisfies \(\bar{u}^2\circ \check{\psi}_1 =
  \bar{u}^1\).
We also found a subsequence \(i\mapsto \ell_i^1\) of \(i\mapsto i\), and a
  subsequence  \(i\mapsto \ell_i^2\) of \(i\mapsto \ell_i^1\) for which  we
  have Gromov convergence
  \begin{equation*}                                                       %% EQN
    \big(u_{\ell_i^k}, \overline{\Sigma}_{\ell_i^k}^k, j_{\ell_i^k},
    \overline{W}, J_{\ell_i^k}, \mu_{\ell_i^k}\cap
    \overline{\Sigma}_{\ell_i^k}^k, D_{\ell_i^k} \cap
    \overline{\Sigma}_{\ell_i^k}^k  \big)
    \to  \big(\bar{u}^k, \overline{\Sigma}^k, \bar{j}^k, \overline{W},
    \overline{J}, \mu^k, D^k  \big),
    \end{equation*}
  for \(k\in \{1, 2\}\), and where all curves in the sequence and limit
  are compact marked nodal boundary-immersed pseudoholomorphic curves.
The final, and most important observation to make is that the
  construction to obtain these results is iterative, and hence for each
  \(k\in \mathbb{N}\) we can construct \(\Sigma^k\), \(\check{\psi}_k\),
  \(\check{\Sigma}^k\), \(i\mapsto \ell_i^k\) etc, with all the associated
  properties.
Observe that the holomorphic maps \(\check{\psi}_k:(\Sigma^k, \bar{j}^k)
  \to (\Sigma^{k+1}, \bar{j}^{k+1})\) send marked points to marked
  points and nodal points to nodal points, and hence give rise to an
  embedding diagram in the sense of Section
  \ref{SEC_direct_limit_manifolds}.
Consequently there exists a direct limit manifold \(\overline{S}\) and 
  holomorphic embeddings \(\check{\iota}_k:
  \Sigma^k\to \overline{S}\) with marked points defined by \(\bar{\mu} :=
  \bigcup_{k\in \mathbb{N}} \check{\iota}_k(\bar{\mu}^k)\)  and nodal points
  defined by \(\bar{\mu} := \bigcup_{k\in \mathbb{N}}
  \check{\iota}_k(\overline{D}^k)\).
This will be our limit marked nodal Riemann surface \((\overline{S},
  \bar{j}, \bar{\mu}, \overline{D})\) without boundary.  
We also note that \(\bar{u}^{k+1}\circ \check{\psi}_k = \bar{u}^k\), and
  hence by property (DL-1) in Section \ref{SEC_direct_limit_manifolds} we
  see that the \(\{ \bar{u}^k\}_{k\in \mathbb{N}}\) induce a map \(\bar{u}:
  \check{\Sigma}\to \overline{W}\) for which the tuple
  \begin{equation*}                                                       %% EQN
    \big(\bar{u}, \overline{S}, \bar{j}, \overline{W}, \overline{J},
    \bar{\mu}, \overline{D}\big)
    \end{equation*}
  is a proper marked nodal pseudoholomorphic curve.  

At this point we have constructed our limit curve, however it still
  remains to show that (after passing to a diagonal subsequence) we have the
  desired exhaustive Gromov convergence.
To that end, we next define the diagonal subsequence \(i\mapsto
  \delta_i:=\ell_i^i\).
Recall that our iterative construction yielded the sequence of smooth
  two-dimensional manifolds with boundary
  \(\check{\Sigma}^k\subset\overline{\Sigma}^k\subset S_k\).
We now define
  \(\overline{S}^{\delta_k}:=\check{\iota}_{\delta_k}
  (\check{\Sigma}^{\delta_k})\subset \overline{S}\) for each \(k\in
  \mathbb{N}\).
We next aim to define \(S_{\delta_k}^{\delta_\ell} \subset S_{\delta_k}\)
  for all \(k, \ell\in \mathbb{N}\) with \(k\geq \ell\).
The precise definition is given as
  \begin{equation*}                                                       %% EQN
    S_{\delta_k}^{\delta_\ell}:={\rm
    cl}\Big(\phi_{\delta_k}^{\delta_\ell}\big(
    \check{\iota}_{\delta_\ell}^{-1}(\overline{S}^{\delta_\ell}\setminus
    \overline{D})\big)\Big),
    \end{equation*}
  where the \(\phi_\cdot^\cdot\) are the diffeomorphisms guaranteed by Gromov
  convergence from the blown up limit Riemann surface to the blown up
  approximating Riemann surfaces:
  \begin{align*}                                                          %% EQN
    \phi_{\ell_i^1}^1&:(\overline{\Sigma}^1)^{D^1,r^1} \to
    (\overline{\Sigma}_{\ell_i^1}^1)^{D_{\ell_i^1},r_{\ell_i^1}}
    \\
    \phi_{\ell_i^2}^2&:(\overline{\Sigma}^2)^{D^2,r^2} \to
    (\overline{\Sigma}_{\ell_i^2}^2)^{D_{\ell_i^2},r_{\ell_i^2}}
    \\
    \phi_{\ell_i^3}^3&:(\overline{\Sigma}^3)^{D^3,r^3} \to
    (\overline{\Sigma}_{\ell_i^3}^3)^{D_{\ell_i^3},r_{\ell_i^3}}
    \\
    &\; \vdots
    \end{align*}

For the sake of clarity, we also provide a more geometric description of
  the definition of the \(S_{\delta_k}^{\delta_\ell}\).  
Start with the compact manifold with boundary
  \(\overline{S}^{\delta_\ell}\subset \overline{S}\).
Then remove the nodal points to obtain
  \(\overline{S}^{\delta_\ell}\setminus \overline{D}\).
This lies in the image of the embedding \(\check{\iota}_{\delta_\ell}:
  \check{\Sigma}^{\delta_\ell}\to \overline{S}\);
  recall that the \(\check{\Sigma}^{\delta_\ell}\) (which satisfy
  \(\check{\Sigma}^{\delta_\ell}\subset
  \overline{\Sigma}^{\delta_\ell}\subset S_{\delta_\ell} \)) form the
  embedding sequence which has direct limit \(\overline{S}\),
  and the \(\check{\iota}_{\delta_\ell}:\check{\Sigma}_{\delta_\ell}\to
  \overline{S}\) are the associated embeddings (which can be thought of as
  inclusions since their images exhaust \(\overline{S}\)).
Consequently, we pull back via
  \(\check{\iota}_{\delta_\ell}\) to obtain a subset of
  \(\overline{\Sigma}^{\delta_\ell}\).
Because we have removed the nodal points, we may identify
  \(\overline{\Sigma}^{\delta_\ell}\setminus D^{\delta_\ell}\) with the
  circle blown-up manifold with special circles removed:
  \((\overline{\Sigma}^{\delta_\ell})^{D^{\delta_\ell},
  r^{\delta_\ell}}\setminus \cup_i \Gamma_i\).
Consequently, we may then map our set via the
  \(\phi_{\delta_k}^{\delta_\ell}\) to obtain
  \begin{equation*}                                                       %% EQN
    \phi_{\delta_k}^{\delta_\ell}\big(
    \check{\iota}_{\delta_\ell}^{-1}(\overline{S}^{\delta_\ell}\setminus
    \overline{D})\big)
\subset \overline{\Sigma}_{\delta_k}^{\delta_\ell} \setminus D_{\delta_k}  
\subset \widetilde{S}_{\delta_k}
\subset S_{\delta_k}
    \end{equation*}
  which would be the compact manifolds with smooth boundaries that we
  seek, except that these sets are missing the images of the special
  circles.
Thus after taking the closure of these sets, we obtain the desired compact
  sets, which we have denoted \(S_{\delta_k}^{\delta_\ell}\).

Thus, we have passed to a subsequence \(\{\mathbf{u}_{\delta_i}\}_{i\in
  \mathbb{N}}\), and constructed a proper marked nodal pseudoholomorphic
  curve \((\bar{u}, \overline{S}, \bar{j}, \overline{W}, \overline{J},
  \bar{\mu}, \overline{D})\), and found smooth compact two-dimensional
  manifolds with boundary \(\overline{S}^{\delta_\ell}\subset \overline{S}
  \) and \(S_{\delta_k}^{\delta_\ell} \subset S_{\delta_k}\) which we now
  claim have the following properties by construction.
\begin{enumerate}                                                         %% NUM
  \item 
  \(\overline{S}^\ell\subset \overline{S}^{\ell+1}\setminus \partial
  \overline{S}^{\ell+1}\) for all \(\ell\in \mathbb{N}\)
  \item 
  \(\overline{S} = \bigcup_{\ell \in \mathbb{N}} \overline{S}^\ell\) 
  \item
  for each fixed \(k\in \mathbb{N}\) and each \(0\leq \ell \leq k-1\) we
  have \(S_k^\ell\subset S_k^{\ell+1}\setminus \partial S_k^{\ell+1}\)
  \item
  for each \(k\geq  \ell\in \mathbb{N}\) we have
  \begin{equation*}                                                       %% EQN
    u_{k} ^{-1}(W_{\ell})\subset S_k^{\ell},
    \end{equation*}
  \item 
  For each fixed \(\ell\in \mathbb{N}\), the sequence 
  \begin{equation*}                                                       %% EQN
    \big\{
    \big(u_k,\; 
    S_k^{\ell},\; 
    j_k,\;  
    \overline{W}, \;
    J_k, \;
    S_k^{\ell}\cap\mu_k,\; 
    S_k^{\ell}\cap D_k\big)
    \big\}_{k \geq \ell}
    \end{equation*}
  is a sequence of compact marked nodal stable boundary-immersed
  pseudoholomorphic curves which converges in a Gromov sense to the
  proper marked nodal stable boundary-immersed pseudoholomorphic curve
  \begin{equation*}                                                       %% EQN
    \big(\bar{u},\; 
    \overline{S}^{\ell},\; 
    \bar{j},\;  
    \overline{W}, \;
    \overline{J}, \;
    \overline{S}^{\ell}\cap\bar{\mu},\; 
    \overline{S}^{\ell}\cap \overline{D}\big).
    \end{equation*}
  \end{enumerate}

The first two properties essentially follow from the fact that \((W_k,
  J_k, g_k)\) is a sequence of properly exhausting regions for
  \((\overline{W}, \overline{J}, \bar{g})\).
To see this, first observe that whenever these two properties hold, they
  will also hold for any subsequence of slightly trimmed compact sets.
Second, recall that \(\widehat{S}_\ell^k =u_\ell^{-1}(W_k)\), and then by
  slightly target-trimming our curves and passing to a subsequence (here
  still denoted with subscripts \(\ell\)) we found compact manifolds with
  boundary \(\overline{\Sigma}_\ell^k\subset \widehat{S}_\ell^{k+1}\) on
  which the subsequence \(\bar{u}_\ell\) still converged.
The domains of these limit curves we denoted \(\overline{\Sigma}^k\);
  we denoted their interiors by \(\Sigma^k\), and we used these and the
  Gromov convergence of the curves to construct holomorphic maps
  \(\check{\psi}_k: \Sigma^k \to \Sigma^{k+1}\) which resulted in an
  associated limit Riemann surface \(\overline{S}\).
The compact manifolds with boundary \(\overline{S}^k\subset \overline{S}\)
  were then obtain by trimming the \(\check{\iota}_k(\Sigma^k)\subset
  \overline{S}\) slightly further.
Because this procedure only involved passing to subsequences and making
  small trimmings, the first two properties follow immediately from the
  definition of \(\overline{S}\).

We establish the third property in a moment, but at present we work on the
  fourth property. 
Indeed, recall that  \( u_k^{-1}(W_\ell) = \widehat{S}_k^\ell\), so that
  we need to verify that \(\widehat{S}_k^\ell\subset S_k^{\ell}\).
Also note that due to the properly exhausting nature of the \(W_k\) it
  follows that \(\widehat{S}_k^\ell \subset \widehat{S}_k^{\ell+1}\) in such
  a way that for any sufficiently small trimming \(\widetilde{\Sigma}\) of
  \(\widehat{S}_k^\ell\), we have \(\widehat{S}_k^\ell \subset
  \widetilde{\Sigma} \subset \widehat{S}_k^{\ell+1}\) .
However, recall that \(\overline{\Sigma}_k^\ell\) was obtained as a small
  trimming of  \(\widehat{S}_k^{\ell+1}\), and \(S_k^\ell\) was obtained as
  a small trimming of \(\overline{\Sigma}_k^\ell\), from which we see that
  \begin{equation*}                                                       %% EQN
    \widehat{S}_k^{\ell} \subset  S_k^\ell\subset
    \overline{\Sigma}_k^\ell\subset\widehat{S}_k^{\ell+1}
    \end{equation*}
  and hence we indeed have \(\widehat{S}_k^{\ell} \subset  S_k^\ell\), as
  desired.
Note however that extending this string of containments a bit further, we
  have
  \begin{equation*}                                                       %% EQN
   S_k^{\ell-1}\subset \overline{\Sigma}_k^{\ell-1}\subset
   \widehat{S}_k^{\ell} \subset  S_k^\ell\subset
   \overline{\Sigma}_k^\ell\subset\widehat{S}_k^{\ell+1},
   \end{equation*}
  which establishes the third property.
Finally, with the \(\overline{S}^k\) obtained as slightly trimmed versions
  of the \(\overline{\Sigma}^k\supset \overline{S}^k\), together with the
  result that we obtained (after passing to a subsequence, which we still
  denote with subscripts \(\ell\)) convergence of the
  \(u_\ell:\overline{\Sigma}_k^\ell\to \overline{W}\) to
  \(\bar{u}^k:\overline{\Sigma}^k\to \overline{W}\), we were able to trim in
  the limit domains \(\overline{\Sigma}^k\), and we were able to use the
  diffeomorphisms \(\phi_k^\ell\) guaranteed by the usual Gromov convergence
  to push forward the \(\overline{S}^\ell\) into the \(S_k\); the images of
  these sets we defined to be \(S_k^\ell\), and they had the property that by
  construction maps \(u_k:S_k^\ell\to \overline{W}\) converged in a
  Gromov sense to \(\bar{u}^\ell:\overline{S}^\ell\to \overline{W}\).
This then establishes the fifth property.

\end{proof}%%%%%%%%%%%%%%%%%%%%%%%%%%%%%%%%%%%%%%%%%%%%%%%%            END PROOF

\appendix

%%%%%%%%%%%%%%%%%%%%%%%%%%%%%%%%%%%%%%%%%%%%%%%%%%%%%%%%%%%%%%%%%%%%%%%%%%%%%%%%
%%%%%%%%%%                           SECTION                           %%%%%%%%%
%%%%%%%%%%%%%%%%%%%%%%%%%%%%%%%%%%%%%%%%%%%%%%%%%%%%%%%%%%%%%%%%%%%%%%%%%%%%%%%%
%
\section{Formula for arithmetic genus}
  \label{SEC_formulat_arithmetic_genus}
Here we provide the more standard formula based definition of arithmetic
  genus, and we show that this is equivalent to the notion provided in
  Definition \ref{DEF_arithmetic_genus}.
Before proceeding to prove that result, we note that we will employ the
  following notation.
For any topological space \(X\), we let \(\pi_0(X)\) denote the set of
  connected components of \(X\), and we let \(\#\pi_0(X)\) denote the number
  of connected components of \(X\).
Additionally we recall the discussion following Definition 
  \ref{DEF_nodal_riemann_surface} in which we defined a topological space
  \(|S|\) associated to a nodal Riemann surface \((S, j, D)\) via
  identifying points in each nodal pair: \(\underline{d}_i\sim
  \overline{d}_i\).
In this way, we will abuse language a bit by saying that \(\Sigma\subset
  S\) is a connected component of \(|S|\) whenever \(|\Sigma|\) is a
  connected component of \(|S|\).

%%%%%%%%%%%%%%%%%%%%%%%%%%%%%%%%%%%%%%%%%%%%%%%%%%%%%%%%%%%%%%%%%%%%%%%%%%%%%%%%
%%%%%%%%%%                            LEMMA                            %%%%%%%%%
%%%%%                                                                       %%%%
\begin{lemma}[formula for arithmetic genus]\hfill\\
  \label{LEM_formula_for_arithmetic_genus}
Let \((S, j, D)\) be a compact marked nodal Riemann surface, possibly with
  boundary.
Then a formula for the arithmetic genus of \((S, j, D)\) is given by
  \begin{equation*}                                                       %% EQN
    {\rm Genus}_{arith}(\mathbf{u})=\#\pi_0(|S|)-\#\pi_0(S)
    +\Big(\sum_{k=1}^{\# \pi_0(S)}
    g_k\Big) +{\textstyle \frac{1}{2}}\# D.
    \end{equation*}
\end{lemma}
%%%%%                                                                       %%%%
%%%%%%%%%%                                                             %%%%%%%%%
%%%%%%%%%%%%%%%%%%%%%%%%%%%%%%%%%%%%%%%%%%%%%%%%%%%%%%%%%%%%%%%%%%%%%%%%%%%%%%%%
%
\begin{proof}
For notational convenience, we define
  \begin{equation*}                                                       %% EQN
    m= \#\pi_0(|S|), \qquad n= \#\pi_0(S)\qquad\text{and}\qquad b =
    \#\pi_0(\partial S).
    \end{equation*}
We begin by denoting the connected components of \(S\) by \(S_k\), so that
  \(\cup_{k=1}^n S_k= S\).
Next, we let \(D_k = S_k\cap D\), and we let \(S_k^{D_k}\) denote the
  circle compactification of \(S_k\setminus D_k\), and we let \(S^{D, r}\)
  denote the surface obtained by gluing the \(S_k^{D_k}\) along pairs of
  compactification circles associated to nodal pairs.
Then, letting \(G_a\) denote the arithmetic genus of \(\mathbf{u}\), and
  letting \(g_k\) denote the genus of \(S_k\), we have
  \begin{align*}                                                          %% EQN
    2m - 2 G_a - b &=   2m - 2{\rm Genus}(S^{D, r}) - b
    \\
    &= \chi(S^{D, r})
    \\
    &=\sum_{k=1}^n \chi(S_k^{D_k})
    \\
    &=\sum_{k=1}^n \big(\chi(S_k) - \# D_k\big)
    \\
    &=\sum_{k=1}^n \big(2 - 2g_k -\#\pi_0 (\partial S_k) - \# D_k\big)
    \\
    &=2n -2\Big(\sum_{k=1}^n g_k\Big) -b -\#D.
    \end{align*}
Solving for \(G_a\), the desired result is immediate.
\end{proof}

\bibliography{bibliography}{}
\bibliographystyle{plain}
\end{document}